\documentclass[10pt,leqno]{amsart}
\usepackage[utf8]{inputenc}
\usepackage{longtable}
\usepackage{varwidth}
\usepackage{amstext}
\usepackage{amsthm}
\usepackage{amssymb}
\usepackage{microtype}
\usepackage[bookmarks=false,
 breaklinks=false,pdfborder={0 0 1},backref=false,colorlinks=false]
 {hyperref}
\hypersetup{pdftitle={New Curvature Operator of the Second Kind Conditions for Rigidity},
 pdfauthor={Xiaolong Li and Peter Petersen},
 pdfsubject={Einstein rigidity and sharp second-kind curvature estimates},
 pdfkeywords={Einstein manifolds, curvature operator of the second kind, Bochner technique, rigidity}}

\makeatletter

\providecommand{\tabularnewline}{\\}
\newenvironment{cellvarwidth}[1][t]
    {\begin{varwidth}[#1]{\linewidth}}
    {\@finalstrut\@arstrutbox\end{varwidth}}

\numberwithin{equation}{section}
\numberwithin{figure}{section}

\@ifundefined{date}{}{\date{}}
\usepackage{amsthm}\usepackage{needspace}% Citation labels are retained without printing an additional References section.
\numberwithin{equation}{section}
\numberwithin{table}{section}
\newtheorem{theorem}{Theorem}[section]
\newtheorem{proposition}[theorem]{Proposition}\newtheorem{corollary}[theorem]{Corollary}\newtheorem{lemma}[theorem]{Lemma}\newtheorem{conjecture}{Conjecture}
\theoremstyle{remark}
\newtheorem{remark}[theorem]{Remark}\newtheorem*{remark*}{Remark}\DeclareMathOperator{\tr}{tr}
\DeclareMathOperator{\ad}{ad}

\DeclareMathOperator{\Id}{Id}
\DeclareMathOperator{\Ric}{Ric}
\newcommand{\owedge}{\mathbin{\wedge}}

\title[New COSK Conditions for Rigidity]{New Curvature Operator of the Second Kind Conditions for Rigidity}
\author[Li]{Xiaolong Li}
\address{Department of Mathematics and Statistics, Auburn University,
Auburn, AL 36849, USA}
\email{xil0005@auburn.edu}
\author[Petersen]{Peter Petersen}
\address{University of California, Los Angeles, 520 Portola Plaza, CA, 90095}
\email{petersen@math.ucla.edu}
\subjclass[2020]{Primary 53C25; Secondary 53C21, 53C35, 58A14}
\keywords{Einstein manifolds, curvature operator of the second kind,
Bochner technique, Weyl tensor, harmonic forms, symmetric spaces, rigidity}

\makeatother

\providecommand\theoremname{Theorem}
\theoremstyle{plain}

\begin{document}
\begin{abstract}
We establish Tachibana-type rigidity theorems with the aid of the
curvature operator of the second kind (COSK). A sharp estimate that
combines the symmetric and commutator actions on Weyl tensors leads
to a constant $C_{n}$, with $\frac{C_{n}}{n}\to\frac{5}{4}$, such
that a complete Einstein manifold of dimension $n\geq5$ with $C_{n}$-nonnegative
COSK has constant nonnegative sectional curvature. We also explore
shifted curvature operators of the second kind leading to similar rigidity results, offer
an elegant Ricci-corrected Bochner identity for harmonic forms, and
classify the complete simply connected locally symmetric Einstein
manifolds in the resulting shifted cone. Explicit spaces establish
the sharpness of the tensor estimates, and complete lists of the eigenvalues of the COSKs for symmetric spaces identify which of them are borderline
cases. 
\end{abstract}

\maketitle

\section{Introduction}

\label{sec:introduction}

The classical theorem of Tachibana \cite{Tachibana} states that
a closed Einstein manifold with nonnegative curvature operator is
locally symmetric. The proof uses a Bochner formula for the curvature
operator. Wink and the second author \cite{PW21} developed refinements in which
nonnegativity of the entire operator is replaced by conditions on
partial sums of its eigenvalues. More recently, the authors \cite{LiPetersen2026} established that closed Einstein manifolds of dimension $n\geq4$ are locally symmetric under
the weaker assumption that the curvature operator is
$\frac{2(n-1)}{3}$-nonnegative. For the curvature operator of the
second kind the analogous problem has a stronger rigidity conclusion:
nonnegativity forces a closed Einstein metric to have constant sectional
curvature; see \cite{Kashiwada93} and \cite{DaiFu}.

Our aim is to obtain rigidity under weaker partial-sum conditions
and to study related conditions after adding Ricci and scalar curvature
terms, i.e., shifting the curvature operator. Rigidity and sphere
theorems under various COSK assumptions were established by Cao-Gursky-Tran
\cite{CGT}, the first author \cite{LiNonnegative,LiFour, Li2025}, Nienhaus-Petersen-Wink \cite{NPW}, Nienhaus-Petersen-Wink-Wylie \cite{NPWW2023}, and Dai-Fu \cite{DaiFu}.
Related partial sum conditions for Einstein metrics were studied by
Fu-Lu \cite{FuLu}, and Fu-Lu-Dai \cite{FuLuDai} and Fu-Lu \cite{FuLuHarmonicWeyl} generalized this
to manifolds with harmonic curvature and harmonic Weyl curvature.

The partial sum conditions for a self-adjoint operator, $T$, with
eigenvalues $\lambda_{1}\leq\cdots\leq\lambda_{N}$ we consider are
of the form 
\[
\Sigma_{k}(T)=\lambda_{1}+\cdots+\lambda_{\lfloor k\rfloor}+(k-\lfloor k\rfloor)\lambda_{\lfloor k\rfloor+1},\,1\leq k\leq N,
\]
where the final term is omitted when \(k\) is an integer. 
We say that $T$ is \emph{k-nonnegative} if $\Sigma_{k}(T)\geq0$
and \emph{$k$-positive} if $\Sigma_{k}(T)>0$. The operators that
satisfy such a condition form a closed, respectively open, cone. The
averages $\frac{\Sigma_{k}(T)}{k}$ are nondecreasing so larger $k$
lead to larger cones. We write $\mathring{R}$ for the curvature operator
of the second kind on $S^{2}_{0}(TM)$, the bundle of trace-free symmetric
two-tensors. It is normalized so that it is the identity on the unit
sphere. For an $n$-dimensional manifold we have that
\[
N=\dim S^{2}_{0}(T_{x}M)=\frac{(n-1)(n+2)}{2}.
\]

The constants needed in the first rigidity theorem are determined by 
\begin{align}
A_{n} & =12-\frac{20}{n}+\frac{16}{n(n-2)^{2}},\nonumber \\
B_{n} & =\frac{8(n-4)}{n(n+2)},\nonumber  \\ \label{eq:intro-constants}
C_{n} & =\frac{3(n-1)}{4}\frac{A_{n}+8}{A_{n}+B_{n}}. 
\end{align}
Asymptotically, $\frac{C_{n}}{n}\to\frac{5}{4}$ as $n\to\infty$.

\begin{theorem}\label{thm:main} Let $(M^{n},g)$ be a complete,
connected Einstein manifold of dimension $n\geq5$. If $\mathring{R}$
is $C_{n}$-nonnegative at every point, then $g$ has constant nonnegative
sectional curvature. \end{theorem}

\begin{remark*}[Dimension four] When $n=4$, $C_4=\frac{9}{2}$ and Li \cite[Theorems~1.5 and~1.9]{LiFour}
showed that a complete, connected Einstein four-manifold with $\frac{9}{2}$-nonnegative
$\mathring{R}$ has constant sectional curvature or is up to scaling
isometric to $\mathbb{CP}^{2}$ with its Fubini-Study metric. This
comes from showing that $\frac{9}{2}$-nonnegativity implies nonnegative
isotropic curvature allowing the use of Micallef-Wang \cite[Theorem~4.4]{MicallefWang1993}.
The projective-space lies on the boundary of the cone, if $\Ric=\frac{1}{2}g$, then the eigenvalues on $\mathbb{CP}^{2}$ are $-\frac{1}{6}$ with multiplicity 3 and $\frac{1}{3}$ with multiplicity 6, so $\Sigma_{\frac{9}{2}}(\mathring{R})=0$; see \cite{LiFour} and Appendix~\ref{app:symmetric-spectra}. Note that $S^{2}\times S^{2}$ is excluded since its $\frac{9}{2}$-sum
is $-\rho<0$ when $\Ric=\rho g$.
\end{remark*}
\begin{remark*}
Theorem \ref{thm:main} resolves the Einstein case of a question of Fu-Lu-Dai \cite[Page 3]{FuLuDai} under weaker nonnegativity assumptions. 
\end{remark*}

The proof of Theorem~\ref{thm:main} combines two complementary actions
of symmetric endomorphisms on the Weyl tensor. Write 
$$R=\kappa I+W,$$
where $I$ has sectional curvature 1 and $\Ric=(n-1)\kappa g=\rho g$.
Symmetric 2-tensors are identified with symmetric endomorphisms using
$g$. For $S\in S^{2}_{0}(T_{x}M)$, define the induced action on
2-vectors by 
\[
\ad_{S}(u\wedge v)=Su\wedge v+u\wedge Sv.
\]
Regarding $W$ as a self-adjoint operator on $\bigwedge^{2}T_{x}M$,
the \emph{symmetric action} and the \emph{commutator action} are 
\begin{align*}
S\cdot W & =\ad_{S}W+W\ad_{S},\\{}
[\ad_{S},W] & =\ad_{S}W-W\ad_{S}.
\end{align*}
The symmetric action applies $S$ with a positive sign in all variables. The commutator applies it positively in the last two and negatively in the first two. More precisely, in an orthonormal basis
with $S=\operatorname{diag}(a_{1},\ldots,a_{n})$, 
\begin{align*}
(S\cdot W)_{ijkl} & =(a_{i}+a_{j}+a_{k}+a_{l})W_{ijkl},\\{}
[\ad_{S},W]_{ijkl} & =(a_{k}+a_{l}-a_{i}-a_{j})W_{ijkl}.
\end{align*}
Section~\ref{sec:preliminaries} establishes the operator and tensor
identifications in detail.

Both actions can be used to represent the curvature term in the Einstein
Bochner formula. With $\Delta=\operatorname{div}\nabla$ and full
tensor norms the relevant Bochner formula is
\[
\frac{1}{2}\Delta|W|^{2}=|\nabla W|^{2}+2Q(R,W).
\]
Select an orthonormal eigenbasis $S_{i}$ of $\mathring{R}$, with
$\mathring{R}(S_{i})=\lambda_{i}S_{i}$ and $\mathring{W}(S_{i})=\mu_{i}S_{i}$,
so that $\mu_{i}=\lambda_{i}-\kappa$. The two representations are
\begin{align*}
2Q(R,W) & =\sum_{i}\lambda_{i}|[\ad_{S_{i}},W]|^{2},\\
6Q(R,W) & =\sum_{i}\lambda_{i}\bigl(|S_{i}\cdot W|^{2}+16\mu^{2}_{i}+B_{n}|W|^{2}\bigr).
\end{align*}
Each expresses the same curvature polynomial as a sum of eigenvalues
with nonnegative weights. The usual comparison argument divides the
total weight by an upper bound for a single weight. Applied separately,
the sharp bounds for these two representations lead to the ratios
\[
\frac{3(n-1)}{4},\quad\frac{6(n-1)}{A_{n}+B_{n}},
\]
whose leading terms are $\frac{3n}{4}$ and $\frac{n}{2}$.

The new ingredient is a sharp \emph{joint} estimate. The key point
is that a large commutator contribution forces the corrected symmetric
contribution to be small, and conversely. Estimating them together
allows for improved estimates on the weights without sacrificing the
total weight. More precisely, Proposition~\ref{prop:joint-estimate}
shows that the combined weights 
\[
w_{i}=\frac{3A_{n}}{8}|[\ad_{S_{i}},W]|^{2}+|S_{i}\cdot W|^{2}+16\mu^{2}_{i}+B_{n}|W|^{2}
\]
satisfy
\begin{align*}
0\leq w_{i} & \leq(A_{n}+B_{n})|W|^{2},\\
\sum_{i}w_{i} & =\frac{3(n-1)(A_{n}+8)}{4}|W|^{2}.
\end{align*}
Adding the separate estimates would give the larger upper bound $(2A_{n}+B_{n})|W|^{2}$.
The joint bound increases the ratio of total weight to the individual
bound to $C_{n}$, with leading term $\frac{5n}{4}$. This yields
a substantial improvement over either representation used alone.

Combining the two identities in these proportions and applying
weight comparison gives 
\[
\frac{1}{2}\Delta|W|^{2}\geq|\nabla W|^{2}+\frac{8(A_{n}+B_{n})}{3(A_{n}+8)}\Sigma_{C_{n}}(\mathring{R})|W|^{2}.
\]
The eigenvalue assumption forces either flatness or a positive Einstein
constant. In the latter case Bonnet-Myers gives compactness, and integration
makes $W$ parallel, so the metric is locally symmetric. A more detailed
analysis of the eigenvalues shows that the space lies in the open
cone for the shifted operator in Theorem~\ref{thm:symmetric-classification}.
However, only the round sphere satisfies that strict condition.

There are two algebraic estimates underlying this argument. The first
gives the sharp bound for the symmetric contribution.

\begin{proposition}[Sharp symmetric-action estimate]\label{prop:sharp-four-slot}
Let $V$ be an $n$-dimensional real Euclidean space, $n\geq4$. For
every algebraic Weyl tensor $W$ on $V$ and every trace-free symmetric
endomorphism $S\in S^{2}_{0}(V)$, 
\begin{equation}
|S\cdot W|^{2}+16|\mathring{W}(S)|^{2}\leq A_{n}|S|^{2}|W|^{2},\label{eq:intro-sharp}
\end{equation}
where 
\[
A_{n}=12-\frac{20}{n}+\frac{16}{n(n-2)^{2}}.
\]
The constant $A_{n}$ is optimal in every dimension, even when $S$
is restricted to eigenvectors of $\mathring{W}$. \end{proposition}

We prove Proposition~\ref{prop:sharp-four-slot} in Section~\ref{sec:sharp-estimates}.
The symmetric and commutator actions satisfy the following stronger
joint estimate.

\begin{proposition}[Sharp joint estimate]\label{prop:joint-estimate}
Let $V$ be an $n$-dimensional real Euclidean space, $n\geq4$. For
every algebraic Weyl tensor $W$ on $V$ and every trace-free symmetric
endomorphism $S\in S^{2}_{0}(V)$, 
\begin{equation}
\frac{3A_{n}}{8}|[\ad_{S},W]|^{2}+|S\cdot W|^{2}+16|\mathring{W}(S)|^{2}\leq A_{n}|S|^{2}|W|^{2}.\label{eq:intro-joint}
\end{equation}
Equality is attained with either the commutator contribution or the corrected symmetric contribution on the left vanishing, even
when $S$ is restricted to eigenvectors of $\mathring{W}$. \end{proposition}

We prove Proposition~\ref{prop:joint-estimate} in Section~\ref{sec:joint-estimate}.

For general metrics, a natural source of shifted second-kind curvature
operators is the family of equivariant linear maps introduced by Böhm-Wilking
\cite{BohmWilking}. With $s=\tr\Ric$ and $E=\Ric-\frac{s}{n}g$,
our normalization is 
\[
\ell_{a,b}(R)=R+bE\owedge g+\frac{as}{n}g\owedge g.
\]
We denote the curvature operator of the second kind induced by this
map by 
\[
\mathring{R}_{a,b}:=\mathring{\ell_{a,b}(R)}.
\]
Based on the results in \cite{NPW} we consider the parameters:
\begin{equation}
\gamma_{n}=\frac{2}{n+4},\quad\mathring{R}_{\gamma_{n}}:=\mathring{R}_{\frac{1}{n+2},\,\gamma_{n}},\label{eq:intro-shift-parameters}
\end{equation}
where 
\begin{equation}
\mathring{R}_{\gamma_{n}}=\mathring{R}+\frac{2}{n+4}J_{E}+\frac{2s}{n(n+2)}\Id.\label{eq:intro-shifted-operator}
\end{equation}
Here $J_{E}(S)=\pi_{0}(ES+SE)$ and $\pi_{0}$ is orthogonal projection
onto $S^{2}_{0}(TM)$. 

\begin{theorem}\label{thm:shifted-harmonic} Let $(M^{n},g)$ be
a closed, connected Riemannian manifold, $n\geq3$. If $\mathring{R}_{\gamma_{n}}$
is $\frac{n+2}{2}$-nonnegative at every point, then every harmonic
differential form is parallel. If the inequality is strict at one
point, then 
\[
H^{p}(M;\mathbb{R})=0,\quad1\leq p\leq n-1.
\]
In particular, an oriented manifold which is $\frac{n+2}{2}$-positive is a rational homology sphere. \end{theorem}

The reason for the choice of $\mathring{R}_{\gamma_{n}}$ is an exact
cancellation in the Bochner formula of Nienhaus-Petersen-Wink \cite{NPW}.
Its separate Ricci and scalar-curvature terms are absorbed into \eqref{eq:intro-shifted-operator}.
For a harmonic $p$-form $\eta$, the formula becomes the single weighted
sum 
\begin{equation}
\frac{1}{2}\Delta|\eta|^{2}=|\nabla\eta|^{2}+\frac{2}{3}\sum_{i}\lambda_{i}|S_{i}\cdot\eta|^{2},\label{eq:intro-harmonic-bochner}
\end{equation}
where $S_{i}$ is an orthonormal eigenbasis of $\mathring{R}_{\gamma_{n}}$
with eigenvalues $\lambda_{i}$, and $S_{i}\cdot\eta$ is the sum
of the positive actions of $S_{i}$ in the $p$ entries. The ratio
of the total weight to the sharp bound on one weight is exactly $\frac{n+2}{2}$,
explaining the nonnegativity condition in the theorem and the conjectures
below.

For an Einstein metric with $\Ric=(n-1)\kappa g$, the operator simplifies
to 
\[
\mathring{R}_{\gamma_{n}}=\mathring{W}+\beta_{n}\Id,\quad\beta_{n}=\frac{3n\kappa}{n+2}.
\]
These shifts lead to the following two conjectures.
\begin{conjecture}
\label{conj:sphere} Let $(M^{n},g)$ be a closed, connected Riemannian
manifold, $n\geq5$. If $\mathring{R}_{\gamma_{n}}$ is $\frac{n+2}{2}$-positive
at every point, then $M$ is diffeomorphic to a spherical space form. 
\end{conjecture}

\begin{conjecture}
\label{conj:einstein} Let $(M^{n},g)$ be a closed, connected Einstein
manifold, $n\geq5$, with $\Ric=(n-1)\kappa g$. If $\mathring{W}+\beta_{n}\Id$
is $\frac{n+2}{2}$-nonnegative at every point, then $g$ is locally
symmetric. If the spectral inequality is strict at one point, then
$g$ has constant positive sectional curvature. 
\end{conjecture}

Round metrics satisfy the strict condition while complex projective
spaces and Einstein products of two round spheres are borderline cases. The
following results provide some evidence for these conjectures.

Set 
\[
\alpha_{n}=\frac{2(n-1)(n+8)}{nA_{n}}>\frac{(n+8)(n-1)^{2}}{4(n-2)(2n-1)}.
\]
\begin{theorem}\label{thm:shifted-improved} Let $(M^{n},g)$ be
a complete, connected Einstein manifold, $n\geq4$, with $\Ric=(n-1)\kappa g$.
If 
$$\mathring{W}+\frac{3n\kappa}{n+8}\Id$$
is $\alpha_{n}$-nonnegative
at every point, then $g$ has constant nonnegative sectional curvature.
\end{theorem} 

The proof in Section~\ref{sec:shifted} first shows that the metric
is locally symmetric and excludes the nonround models by inspecting
their eigenvalues. 

\begin{theorem}\label{thm:symmetric-classification} Let $(M^{n},g)$
be a complete, simply connected, locally symmetric Einstein manifold,
$n\geq5$, with $\Ric=(n-1)\kappa g$. Then $\mathring{W}+\beta_{n}\Id$
is $\frac{n+2}{2}$-nonnegative if and only if, up to constant rescaling,
$(M,g)$ is Euclidean space, a round sphere, complex projective space
with its Fubini-Study metric, or $S^{p}\times S^{q}$ with $p,q\geq2$,
$p+q=n$, and round metrics for each factor that have the same positive Einstein
constant. The round sphere lies in the interior of the cone while the
other examples are on the boundary. \end{theorem}

One can in fact show that the Ricci curvature is bounded from below
\[
\Ric\geq\frac{2(n-1)}{3n}\Sigma_{\frac{n+2}{2}}\bigl(\mathring{R}_{\gamma_{n}}\bigr)g.
\]
Consequently, a closed, connected manifold of dimension three or four
for which $\mathring{R}_{\gamma_n}$ is
$\frac{n+2}{2}$-positive everywhere is diffeomorphic to a spherical
space form; see Proposition~\ref{prop:shifted-ricci}. In dimension four, the argument also uses the corresponding estimate for two-forms to obtain positive isotropic curvature. 

This paper is organized as follows. Sections~\ref{sec:preliminaries}-\ref{sec:joint-estimate} establish
the curvature identities and sharp algebraic estimates, and Section~\ref{sec:main-proof}
proves Theorem~\ref{thm:main}. Section~\ref{sec:shifted} develops
the Böhm-Wilking shifts, derives \eqref{eq:intro-harmonic-bochner},
and proves the harmonic-form and shifted Einstein rigidity theorems,
together with the supporting Ricci and topological statements. Section~\ref{sec:symmetric-classification}
classifies the locally symmetric models in the appropriate cone of
the shifted nonnegativity assumption. Appendix~\ref{app:symmetric-spectra}
calculates the eigenvalues of the COSK for compact irreducible symmetric
spaces and the Einstein product models. It also corrects a factor-of-two
error in the type-$A$ group entry of Koiso's spectral table \cite{Koiso1980}.

\textbf{AI disclosure.} OpenAI’s ChatGPT assisted the authors in developing this paper, particularly with the representation-theoretic calculations and spectral tables in Appendix~\ref{app:symmetric-spectra}. Complete proofs are included, and all arguments and calculations can be verified by hand. The authors take full responsibility for the mathematical content.

\textbf{Acknowledgments.} The authors thank the American Institute of Mathematics for its hospitality during the May 2026 workshop \emph{The Bochner Technique}, where their collaboration began. They also thank Lei Ni for suggesting the study of shifted curvature operators of the second kind.

\section{Preliminaries}\label{sec:preliminaries}

\subsection{Curvature tensors and norm conventions}

Let $(V,g)$ be an $n$-dimensional real Euclidean space. An algebraic
curvature tensor $R$ is a 4-tensor satisfying 
\[
R(x,y,z,w)=-R(y,x,z,w)=-R(x,y,w,z)=R(z,w,x,y)
\]
and the first Bianchi identity 
\[
R(x,y,z,w)+R(y,z,x,w)+R(z,x,y,w)=0.
\]
The induced inner product on $\bigwedge^{2}V$ is determined by 
\[
\langle x\wedge y,z\wedge w\rangle=\langle x,z\rangle\langle y,w\rangle-\langle x,w\rangle\langle y,z\rangle.
\]
As in \cite[Section~2]{LiPetersen2026}, we use the same letter for
the associated self-adjoint operator, defined by 
\[
\langle R(x\wedge y),z\wedge w\rangle=R(x,y,z,w).
\]
Thus $e_{i}\wedge e_{j}$, $i<j$, is an orthonormal basis when the
$e_{i}$ are orthonormal in $V$. The unit-curvature tensor $I$ and
the Ricci contraction are 
\[
I_{ijkl}=\delta_{ik}\delta_{jl}-\delta_{il}\delta_{jk}
\]
and 
\[
\Ric_{R}(x,y)=\sum_{i}R(x,e_{i},y,e_{i}).
\]
In particular, $I$ is the identity on $\bigwedge^{2}V$ and $\Ric_{I}=(n-1)g$.
An algebraic Weyl tensor is an algebraic curvature tensor $W$ with
vanishing Ricci contraction, $\Ric_{W}=0$.

We distinguish the full tensor norm $|R|$ from the Hilbert-Schmidt
norm $\|R\|$ of the associated curvature operator: 
\begin{equation}
|R|^{2}=\sum_{i,j,k,l}R^{2}_{ijkl}=4\|R\|^{2}=4\tr_{\bigwedge^{2}V}(R^{2}).\label{eq:norm-conventions}
\end{equation}
The relation $|T|^{2}=4\|T\|^{2}$ holds more generally for every
endomorphism $T$ of $\bigwedge^{2}V$. Norms of symmetric endomorphisms
of $V$ are $|S|^{2}=\tr_{V}(S^{2})$. The sharp estimates in this
paper use full tensor norms. We retain the Hilbert-Schmidt normalization
of the curvature term $Q$ in \cite{LiPetersen2026}.

\subsection{The curvature operator of the second kind}

We identify symmetric 2-tensors with self-adjoint endomorphisms using
$g$, and equip $S^{2}(V)$ with the inner product $\langle S,T\rangle=\tr_{V}(ST)$.
The trace free symmetric tensors are denoted by
\[
S^{2}_{0}(V)=\{S\in S^{2}(V):\tr_{V}S=0\},
\]
and its dimension is $N=\dim S^{2}_{0}(V)=\frac{(n-1)(n+2)}{2}$.
The curvature operator of the second kind is the self-adjoint operator
$\mathring{R}:S^{2}_{0}(V)\to S^{2}_{0}(V)$ characterized by 
\begin{equation*}
\langle\mathring{R}(S),T\rangle=\sum_{i,j}R(e_{i},e_{j},Se_{j},Te_{i}),\quad S,T\in S^{2}_{0}(V). %\label{eq:second-kind-definition}
\end{equation*}
The curvature symmetries imply that this expression is symmetric in $S,T$, thus leading to a self-adjoint operator. This definition also makes the
normalization transparent: for $R=I$ the contraction is $\tr(ST)-(\tr S)(\tr T)=\langle S,T\rangle$.
Thus $\mathring{I}=\Id_{S^{2}_{0}(V)}$.

Let $\pi_{0}$ be orthogonal projection onto $S^{2}_{0}(V)$ and define
\[
(\widehat{R}S)(x,y)=\sum_{i}R(x,e_{i},Se_{i},y).
\]
Then 
\[
\mathring{R}=\pi_{0}\circ\widehat{R}\big|_{S^{2}_{0}(V)},\quad(\widehat{R}S)_{ij}=\sum_{k,l}R_{iklj}S_{kl}.
\]
Since $\tr(\widehat{R}S)=-\langle\Ric_{R},S\rangle$, we also have
\[
\mathring{R}(S)=\widehat{R}(S)+\frac{\langle\Ric_{R},S\rangle}{n}\Id_{V},\quad S\in S^{2}_{0}(V).
\]
The projection is therefore unnecessary for Einstein curvature tensors
and, in particular, for all algebraic Weyl tensors. The scalar curvature
is given by $s=\tr_{g}\Ric_{R}$ and satisfies 
\begin{equation}
\tr_{S^{2}_{0}(V)}\mathring{R}=\frac{n+2}{2n}s.\label{eq:second-kind-trace}
\end{equation}

For $H\in S^{2}(V)$ the Kulkarni-Nomizu product is given by
\[
(H\owedge g)_{ijkl}=H_{ik}g_{jl}+H_{jl}g_{ik}-H_{il}g_{jk}-H_{jk}g_{il}.
\]
If $J_{H}(S)=\pi_{0}(HS+SH)$ for $S\in S^{2}_{0}(V)$, then 
\begin{equation}
\mathring{(H\owedge g)}=J_{H},\quad\Ric_{H\owedge g}=(n-2)H+(\tr H)g,\quad g\owedge g=2I.\label{eq:ricci-correction-conventions}
\end{equation}
Note that in the first identity $\mathring{(H\owedge g)}(S)=HS+SH-\frac{2}{n} \langle H,S\rangle\Id_{V}$
for trace-free $S$. In particular, for $n\geq3$ and $E=\Ric_{R}-\frac{s}{n}g$,
\begin{align*}
R & =W+\frac{1}{n-2}E\owedge g+\frac{s}{n(n-1)}I,\\
\mathring{R} & =\mathring{W}+\frac{1}{n-2}J_{E}+\frac{s}{n(n-1)}\Id.
\end{align*}
For an Einstein tensor with $\Ric_{R}=(n-1)\kappa g$, this becomes
\[
R=\kappa I+W,\quad\mathring{R}=\kappa\Id+\mathring{W}.
\]

\subsection{Induced, symmetric, and commutator actions}

For an endomorphism $A$ of $V$, its induced action on $\bigwedge^{2}V$
is denoted by 
\[
A\cdot(x\wedge y)=\ad_{A}(x\wedge y)=Ax\wedge y+x\wedge Ay.
\]
It satisfies $(\ad_{A})^{*}=\ad_{A^{*}}$. To relate this notation
to \cite{LiPetersen2026}, identify $\bigwedge^{2}V$ with $\mathfrak{so}(V)$
by 
\[
(x\wedge y)z=\langle x,z\rangle y-\langle y,z\rangle x,\quad\langle L,K\rangle=-\frac{1}{2}\tr_{V}(LK).
\]
Under this identification $\ad_{A}(K)=AK+KA^{*}$. Thus for a skew-adjoint
$L$ it is the usual adjoint action $\ad_{L}(K)=[L,K]$, whereas for
a symmetric $S$ it is the symmetric product $\ad_{S}(K)=SK+KS$.
In particular, $\ad_{L}$ is skew-adjoint and $\ad_{S}$ is self-adjoint
on $\bigwedge^{2}V$.

For a symmetric $S$ and an algebraic Weyl tensor $W$, define the
symmetric action by the anti-commutator 
\[
S\cdot W:=\ad_{S}W+W\ad_{S}.
\]
This is the usual action up to a sign. More precisely, as $\ad_{S}$
is self-adjoint 
\begin{align*}
\langle\ad_{S}W(x\wedge y),z\wedge w\rangle & =W(x,y,Sz,w)+W(x,y,z,Sw),\\
\langle W\ad_{S}(x\wedge y),z\wedge w\rangle & =W(Sx,y,z,w)+W(x,Sy,z,w).
\end{align*}
Adding the two formulas yields 
\begin{equation}
(S\cdot W)(v_{1},v_{2},v_{3},v_{4})=\sum^{4}_{r=1}W(v_{1},\ldots,Sv_{r},\ldots,v_{4}).\label{eq:four-slot-action}
\end{equation}

The commutator action is 
\[
[\ad_{S},W]:=\ad_{S}W-W\ad_{S}.
\]
As a 4-tensor, the action in the last two entries has a positive sign
and in the first two entries a negative sign: 
\begin{align*}
[\ad_{S},W](x,y,z,w)={} & W(x,y,Sz,w)+W(x,y,z,Sw)\\
 & -W(Sx,y,z,w)-W(x,Sy,z,w).
\end{align*}
It is skew-adjoint on $\bigwedge^{2}V$, hence skew under interchange
of the two index pairs. It does have the defect of not being naturally
defined as $[\ad_{S},W]$ need not be an algebraic curvature tensor.
When $S=\operatorname{diag}(a_{1},\ldots,a_{n})$, the two actions
are 
\begin{align*}
(S\cdot W)_{ijkl} & =(a_{i}+a_{j}+a_{k}+a_{l})W_{ijkl},\\{}
[\ad_{S},W]_{ijkl} & =(a_{k}+a_{l}-a_{i}-a_{j})W_{ijkl},
\end{align*}
where the coefficients $a_{i}+a_{j}+a_{k}+a_{l}$ and $a_{k}+a_{l}-a_{i}-a_{j}$
are called the multipliers of the respective actions. 

Since $W\ad_{S}=(\ad_{S}W)^{*}$, the parallelogram rule and \eqref{eq:norm-conventions}
imply the relationship 
\begin{equation}
|[\ad_{S},W]|^{2}+|S\cdot W|^{2}=4|\ad_{S}W|^{2}.\label{eq:pair-norm-identity}
\end{equation}

For comparison, if $L$ is skew-adjoint, the ordinary Lie algebra
action on a curvature tensor is 
\begin{align}
L\cdot W & =[\ad_{L},W],\nonumber \\
(L\cdot W)(v_{1},v_{2},v_{3},v_{4}) & =-\sum^{4}_{r=1}W(v_{1},\ldots,Lv_{r},\ldots,v_{4}).\label{eq:skew-action}
\end{align}
Here the commutator is self-adjoint. The distinction between symmetric
$S$ and skew-symmetric $L$ explains the different signs in \eqref{eq:four-slot-action}
and \eqref{eq:skew-action}.

\subsection{The weight principle}

For a self-adjoint operator $T$ on an $N$-dimensional Euclidean
space, the fractional sum introduced in Section~\ref{sec:introduction}
has the variational characterization 
\begin{equation}
\Sigma_{k}(T)=\min\{\tr(PT):0\leq P\leq\Id,\ \tr P=k\}.\label{eq:kyfan}
\end{equation}
Indeed, in an ordered eigenbasis of $T$, the minimum is obtained
by placing weight 1 on the lowest $\lfloor k\rfloor$ eigenvalues
and, when $k$ is not an integer, weight $k-\lfloor k\rfloor$ on
the next. The averages $\frac{\Sigma_{k}(T)}{k}$ are nondecreasing.
In particular, if $1\leq k<N$, $\tr T=0$, and $\Sigma_{k}(T)\geq0$,
then $T=0$.

We will use the following observation; compare \cite{PW21} and \cite[Lemma~4.1]{LiPetersen2026}. 

\begin{lemma}\label{lem:weight-comparison} Suppose $\lambda_{1}\leq\cdots\leq\lambda_{N}$,
$0\leq w_{i}\leq H$, and $\sum_{i}w_{i}=kH$, where $H>0$ and $1\leq k<N$.
Then 
\[
\sum_{i}\lambda_{i}w_{i}\geq H\Sigma_{k}(\lambda).
\]
\end{lemma} 
\begin{proof}
Let $q_{i}$ be the minimizing weights described after \eqref{eq:kyfan},
and put $c=\lambda_{\lceil k\rceil}$. Since $\sum_{i}(\frac{w_{i}}{H}-q_{i})=0$,
\[
\frac{1}{H}\sum_{i}\lambda_{i}w_{i}-\Sigma_{k}(\lambda)=\sum_{i}(\lambda_{i}-c)(\frac{w_{i}}{H}-q_{i})\geq0.
\]
Every summand is nonnegative: for $i<\lceil k\rceil$ both factors
are nonpositive, while for $i>\lceil k\rceil$ both are nonnegative,
and for $i=\lceil k\rceil$ the first factor vanishes. 
\end{proof}

\subsection{The curvature term and its two representations}

We use the curvature term $Q$ with the Hilbert-Schmidt normalization
of \cite{LiPetersen2026}. Let $\eta_{1},\ldots,\eta_{m}$ be an orthonormal
eigenbasis of $R$ on $\bigwedge^{2}V\simeq\mathfrak{so}(V)$, where
$m=\frac{n(n-1)}{2}$ and 
$$R\eta_{a}=\kappa_{a}\eta_{a}.$$ For an
algebraic Weyl tensor $W$, set 
\begin{equation}
Q(R,W)=\sum^{m}_{a=1}\kappa_{a}\|\eta_{a}\cdot W\|^{2}.\label{eq:Q-definition}
\end{equation}
This is linear in $R$, quadratic in $W$, and independent of the
eigenbasis. Equivalently, for any orthonormal basis of $\bigwedge^{2}V$
it is the contraction 
\[
Q(R,W)=\sum_{a,b}\langle R\eta_{a},\eta_{b}\rangle\langle\eta_{a}\cdot W,\eta_{b}\cdot W\rangle_{\mathrm{HS}}.
\]

For $n\geq4$, let $S_{1},\ldots,S_{N}$ be an orthonormal eigenbasis
of $\mathring{W}$, with eigenvalues $\mu_{i}$, and put $\lambda_{i}=\kappa+\mu_{i}$,
the eigenvalues of $\mathring{R}=\kappa\Id+\mathring{W}$. We list
the following known identities in full tensor norms: 
\begin{equation}
\sum_{i}\mu_{i}=0,\quad\sum_{i}\mu^{2}_{i}=\frac{3}{4}|W|^{2},\label{eq:weyl-traces}
\end{equation}
\[
\sum_{i}|S_{i}\cdot W|^{2}=\frac{2(n^{2}+n-8)}{n}|W|^{2},
\]
and 
\begin{equation}
\sum_{i}\mu_{i}|S_{i}\cdot W|^{2}=-36\tr(W^{3})-4\tr((\mathring{W})^{3}).\label{eq:slot-cubic}
\end{equation}
These follow from the contraction formulas in \cite[Proposition~3.2 and Lemma~3.3]{DaiFu};
see also \cite[Proposition~4.1]{FuLuDai}. Here and below, the trace
of $W^{3}$ is taken on $\bigwedge^{2}V$, and the trace of $(\mathring{W})^{3}$
is taken on $S^{2}_{0}(V)$.

Expanding the action in \eqref{eq:Q-definition} and using the curvature
symmetries leads to the cubic expression 
\begin{equation}
Q(\kappa I+W,W)=(n-1)\kappa|W|^{2}-6\tr(W^{3})+2\tr((\mathring{W})^{3}).\label{eq:Q-polynomial}
\end{equation}
On an Einstein manifold with $\Ric=(n-1)\kappa g$, the Bochner identity
is, in Hilbert-Schmidt norms, 
\[
\frac{1}{2}\Delta\|W\|^{2}=\|\nabla W\|^{2}+\frac{1}{2}Q(\kappa I+W,W).
\]
This is the normalization of the curvature term used in \cite[(1.2)]{LiPetersen2026},
with our stated sign convention for the scalar Laplacian. Equivalently,
in full tensor norms, 
\begin{equation}
\frac{1}{2}\Delta|W|^{2}=|\nabla W|^{2}+2Q(\kappa I+W,W).\label{eq:bochner}
\end{equation}
The factor 2 here comes from \eqref{eq:norm-conventions}; it also
appears below when using the COSK.

\begin{proposition}\label{prop:two-representations} Let $W$ be
an algebraic Weyl tensor on an $n$-dimensional Euclidean space, $n\geq4$,
and let $\kappa\in\mathbb{R}$. With the eigenbasis and eigenvalues
above, 
\begin{align}
\sum_{i}\lambda_{i}|[\ad_{S_{i}},W]|^{2} & =2Q(\kappa I+W,W),\nonumber \\
\sum_{i}|[\ad_{S_{i}},W]|^{2} & =2(n-1)|W|^{2}.\label{eq:commutator-identities}
\end{align}
Set $B_{n}=\frac{8(n-4)}{n(n+2)}$ and define the nonnegative weights
\[
\omega_{i}=|S_{i}\cdot W|^{2}+16\mu^{2}_{i}+B_{n}|W|^{2},
\]
where $S_{i}\cdot W=\ad_{S_{i}}W+W\ad_{S_{i}}$. Then we also have
\begin{align}
\sum_{i}\lambda_{i}\omega_{i} &=6Q(\kappa I+W,W),\nonumber \\
\sum_{i}\omega_{i} &=6(n-1)|W|^{2}.\label{eq:second-representation}
\end{align}

\end{proposition}
\begin{proof}
We start by establishing two formulas

\begin{align}
\sum_{i}\mu_{i}\|\ad_{S_{i}}W\|^{2} & =-3\tr(W^{3}),\nonumber \\
\sum_{i}\|\ad_{S_{i}}W\|^{2} & =\frac{(n-2)(n+2)}{4n}|W|^{2}.\label{eq:induced-action-contractions}
\end{align}
For the first, recall that $\ad_{S_{i}}W$ denotes composition of
operators on $\bigwedge^{2}V$ and for an orthonormal basis $\eta_{a}$
of $\bigwedge^{2}V$
\[
\|\ad_{S_{i}}W\|^{2}=\sum_{a}|\ad_{S_{i}}(W(\eta_{a}))|^{2}=\sum_{a}|S_{i}\cdot(W(\eta_{a}))|^{2}.
\]
Moreover, for an algebraic curvature operator the Weitzenböck curvature
on 2-forms reads
\[
Q_{2}(R,\omega)=g\left(\mathrm{Ric}\circ\omega+\omega\circ\mathrm{Ric}-2R(\omega),\omega\right),
\]
where it is understood that $\omega$ is an operator for the compositions.
Combining this with \cite[Proposition~2.1]{NPW} and using the Weyl
tensor as curvature shows that
\[
-3g\left(W(\omega),\omega\right)=\sum\mu_{i}|S_{i}\cdot \omega|^{2}.
\]
Letting $\omega=W\left(\eta_{a}\right)$ and summing over $a$ then
finishes the proof. 

The second unweighted formula follows in a similar
fashion using that 2-forms satisfy the identity
\[
\sum|S_{i}\cdot\omega|^{2}=\frac{(n-2)(n+2)}{n}|\omega|^{2}
\]
from \cite[Lemma 3.7 (a)]{NPW}.

Multiplying \eqref{eq:pair-norm-identity} by $\mu_{i}$ and summing,
then using \eqref{eq:slot-cubic} and \eqref{eq:induced-action-contractions},
yields 
\begin{align*}
\sum_{i}\mu_{i}|[\ad_{S_{i}},W]|^{2} & =-48\tr(W^{3})+36\tr(W^{3})+4\tr((\mathring{W})^{3})\\
 & =-12\tr(W^{3})+4\tr((\mathring{W})^{3})\\
 & =2Q(W,W).
\end{align*}
The unweighted version of the same calculation is 
\begin{align*}
\sum_{i}|[\ad_{S_{i}},W]|^{2} & =\left(\frac{4(n-2)(n+2)}{n}-\frac{2(n^{2}+n-8)}{n}\right)|W|^{2}\\
 & =2(n-1)|W|^{2}.
\end{align*}
Adding $\kappa$ times this identity proves the first line of \eqref{eq:commutator-identities}
by \eqref{eq:Q-polynomial}.

For the symmetric representation, \eqref{eq:weyl-traces}-\eqref{eq:slot-cubic}
yield 
\begin{align*}
\sum_{i}\bigl(|S_{i}\cdot W|^{2}+16\mu^{2}_{i}\bigr) & =\frac{2(n-1)(n+8)}{n}|W|^{2},\\
\sum_{i}\mu_{i}\bigl(|S_{i}\cdot W|^{2}+16\mu^{2}_{i}\bigr) & =-36\tr(W^{3})+12\tr((\mathring{W})^{3})=6Q(W,W).
\end{align*}
Finally, we have
\[
B_{n}N=\frac{4(n-1)(n-4)}{n},\quad\sum_{i}\mu_{i}=0,\quad\sum_{i}\lambda_{i}=N\kappa.
\]
Thus $\sum_{i}\omega_{i}=6(n-1)|W|^{2}$, while $\sum_{i}\mu_{i}\omega_{i}=6Q(W,W)$.
Replacing $\mu_{i}$ by $\lambda_{i}=\kappa+\mu_{i}$ adds $6(n-1)\kappa|W|^{2}$.
Together with \eqref{eq:Q-polynomial}, this proves \eqref{eq:second-representation}. 
\end{proof}

\section{A sharp estimate for the symmetric action}

\label{sec:sharp-estimates}

We prove Proposition~\ref{prop:sharp-four-slot} by decomposing the
Weyl tensor orthogonally into three types of components. The part
involving components of the form $W_{ijij}$ is controlled by Lemma~\ref{lem:sharp-edge-estimate},
whose proof gives an explicit sum-of-squares decomposition.

Throughout this section, $V$ is an $n$-dimensional Euclidean space,
$n\geq4$, and $W$ is an algebraic Weyl tensor. We identify symmetric
two-tensors with symmetric endomorphisms. All curvature norms are
full tensor norms, while $|S|^{2}=\tr(S^{2})$. Recall that
\begin{align*}
A_{n} & =12-\frac{20}{n}+\frac{16}{n(n-2)^{2}}=8+\frac{4(n-1)(n-4)^{2}}{n(n-2)^{2}}.
\end{align*}

\subsection{An orthogonal decomposition of the Weyl tensor}

The decomposition below is relative to an orthonormal eigenbasis of
$S$. It groups components according to the number of distinct indices,
with the three-index components further grouped by the two indices
that occur only once. This grouping keeps the Ricci equations within
the individual summands.
The support of a tensor $T_{i_1\cdots i_l}$  is the set $ \{(i_1,..., i_l)\mid T_{i_1\cdots i_l}\neq 0 \}  $.  If two tensors of the same type have disjoint supports, then they are clearly perpendicular.

\begin{lemma}\label{lem:weyl-component-decomposition} Fix an orthonormal
basis in which $S=\operatorname{diag}(a_{1},\ldots,a_{n})$. Every
algebraic Weyl tensor $W$ has the decomposition 
\begin{equation}
W=W^{[2]}+\sum_{j<k}W^{[3]}_{jk}+\sum_{\substack{I\subset\{1,\ldots,n\}\\
|I|=4
}
}W^{[4]}_{I},\label{eq:weyl-component-decomposition}
\end{equation}
where each summand retains the specified components of $W$ while
all other components vanish: 
\begin{enumerate}
\item $W^{[2]}$  is one term. It comprises the components of $W$ that use two distinct indices.
\item $W^{[3]}_{jk}$  is one of  $\binom{n}{2}$ terms. It comprises the components of $W$ in which $j$ and $k$ each
occur once and $i \notin \{j,k\}$ occurs twice. 
\item $W^{[4]}_{I}$ is one of $\binom{n}{4}$ terms. It comprises the components of $W$ whose four indices are all four
elements of $I$. 
\end{enumerate}
Every summand is an algebraic Weyl tensor. The terms are mutually
orthogonal, as are their symmetric actions by $S$ and evaluations
of the corresponding COSK operators on $S$. \end{lemma}
\begin{proof}
Antisymmetry in each pair of indices implies that a nonzero component
has multiplicities $(2,2)$, $(2,1,1)$, or $(1,1,1,1)$. These patterns
give exactly the terms in \eqref{eq:weyl-component-decomposition}.
The curvature symmetries and every first Bianchi identity involve
only permutations of one fixed collection of four indices, with repetitions
counted. All their terms therefore lie in the same summand. Thus each
summand retains the algebraic curvature identities.

It remains to check the Ricci contractions. For $i\ne j$, put $x_{ij}=W_{ijij}$
and set $x_{ii}=0$. The only possibly nonzero Ricci entries of $W^{[2]}$
are the diagonal entries, and these vanish because 
\[
(\Ric_{W^{[2]}})_{ii}=\sum_{j\ne i}x_{ij}=(\Ric_{W})_{ii}=0.
\]
For fixed $j<k$, put $y_{i}=W_{ijik}$, with $i\notin\{j,k\}$. The
only possibly nonzero Ricci entries of $W^{[3]}_{jk}$ are the $(j,k)$
and $(k,j)$ entries. Both are calculated by 
\[
\sum_{i\notin\{j,k\}}y_{i}=(\Ric_{W})_{jk}=0.
\]
Each $W^{[4]}_{I}$ has zero Ricci contraction automatically, since
every contraction repeats an index. Hence all the terms are Weyl
tensors.

A fixed set of $4$ indices $\left(i,j,k,l\right)$ can only lie in the support of one of the possible term. Thus the terms are orthogonal. In the chosen basis, 
\[
(S\cdot W)_{ijkl}=(a_{i}+a_{j}+a_{k}+a_{l})W_{ijkl},
\]
so their symmetric actions will also have mutually disjoint supports. Finally,
since $S$ is diagonalized we have
\[
\mathring{W}(S)_{uv}=\sum_{r}a_{r}W_{urrv}.
\]
For $W^{[2]}$, this is a diagonal matrix with $i$th entry $-\sum_{j}a_{j}x_{ij}$.
For $W^{[3]}_{jk}$, only the $(j,k)$ and $(k,j)$ entries can be
nonzero, and both equal $-\sum_{i\notin\{j,k\}}a_{i}y_{i}$. For $W^{[4]}_{I}$,
the expression vanishes. Thus for any pair of indices $\left(u,v\right)$
only one of the terms can be nonzero, proving the last orthogonality
claim. 
\end{proof}

\subsection{A matrix inequality}

The components of $x_{ij}=W_{ijij}$ in an eigenbasis of $S$ form
a symmetric matrix with zero diagonal and zero row sums. The following
lemma gives the required estimate for this matrix.

\begin{lemma}\label{lem:sharp-edge-estimate} Let $n\geq4$. Suppose
that $a_{1},\ldots,a_{n}$ and $x_{ij}$, $1\leq i,j\leq n$, are
real numbers satisfying 
\[
\sum_{i}a_{i}=0,\quad x_{ij}=x_{ji},\quad x_{ii}=0,\quad\sum_{j}x_{ij}=0\quad\text{for every }i.
\]
%Set 
%\[
%\ell_{n}=3-\frac{5}{n}+\frac{4}{n(n-2)^{2}}.
%\]
Then
\[
\sum_{i<j}(a_{i}+a_{j})^{2}x^{2}_{ij}+\sum_{i}\left(\sum_{j}a_{j}x_{ij}\right)^{2}\leq \frac{A_n}{4}\left(\sum_{i}a^{2}_{i}\right)\left(\sum_{i<j}x^{2}_{ij}\right).
\]
Equality implies 
\begin{equation}
(a_{i}-a_{j})x_{ij}=\frac{1}{n}\sum_{k}a_{k}(x_{jk}-x_{ik}), \text{ for } i\ne j.\label{eq:equality-condition}
\end{equation}
\end{lemma}
\begin{proof}
If all $a_{i}$ or all $x_{ij}$ vanish, both assertions are immediate.
Otherwise, homogeneity allows us to assume 
\[
\sum_{i}a^{2}_{i}=1,\quad\sum_{i<j}x^{2}_{ij}=1.
\]
Put $v_{i}=\sum_{j}a_{j}x_{ij}$ and define the quantities 
\[
F=\sum_{i<j}(a_{i}+a_{j})^{2}x^{2}_{ij},\quad G=\sum_{i<j}(a^{2}_{i}+a^{2}_{j})x^{2}_{ij},\quad H=\sum_{i}v^{2}_{i}.
\]
Thus the desired inequality is $F+H\leq\frac{A_n}{4}$. The zero row sums
give 
\[
\sum_{i}v_{i}=0,\quad\sum_{j}(a_{i}-a_{j})x_{ij}=-v_{i}.
\]
To record the square that will give the equality condition, set 
\[
M=\sum_{i<j}\left((a_{i}-a_{j})x_{ij}-\frac{v_{j}-v_{i}}{n}\right)^{2}.
\]
Since $(a_{i}-a_{j})x_{ij}$ is skew-symmetric in $i,j$, we have
\begin{align*}
\sum_{i<j}(a_{i}-a_{j})x_{ij}(v_{j}-v_{i}) & =-\sum_{i}v_{i}\sum_{j}(a_{i}-a_{j})x_{ij}=H,\\
\sum_{i<j}(v_{j}-v_{i})^{2} & =n\sum_{i}v^{2}_{i}-\left(\sum_{i}v_{i}\right)^{2}=nH.
\end{align*}
Expanding $M$, and using $(a_{i}-a_{j})^{2}=2(a^{2}_{i}+a^{2}_{j})-(a_{i}+a_{j})^{2}$,
therefore gives 
\begin{equation}
M=\sum_{i<j}(a_{i}-a_{j})^{2}x^{2}_{ij}-\frac{H}{n}=2G-F-\frac{H}{n}.\label{eq:gradient-projection}
\end{equation}

\smallskip{}
\noindent \textbf{The case $n=4$.} The zero row sums imply 
\[
x_{12}=x_{34}=r,\quad x_{13}=x_{24}=s,\quad x_{14}=x_{23}=t,\quad r+s+t=0.
\]
Normalization gives $2(r^{2}+s^{2}+t^{2})=1$, so every row has squared
norm $\frac{1}{2}$ and hence $G=\frac{1}{2}$. Moreover, 
\[
v_{1}+v_{2}=(a_{1}+a_{2})r+(a_{3}+a_{4})(s+t)=2(a_{1}+a_{2})x_{12}.
\]
The same calculation for each pair shows that $(a_{i}+a_{j})x_{ij}=\frac{v_{i}+v_{j}}{2}$.
Consequently, 
\[
F=\frac{1}{4}\sum_{i<j}(v_{i}+v_{j})^{2}=\frac{1}{4}\left(2\sum_{i}v^{2}_{i}+\left(\sum_{i}v_{i}\right)^{2}\right)=\frac{H}{2}.
\]
Since $A_{4}=8$, equation~\eqref{eq:gradient-projection} now
yields 
\[
2-F-H=2M.
\]
This proves the inequality and shows that equality forces $M=0$.

\smallskip{}
\noindent \textbf{The case $n>4$.} For distinct $i,j,k$, define 
\begin{align*}
p_{ijk}={} & (n-4)(a_{i}x_{jk}+a_{j}x_{ik}+a_{k}x_{ij}) -\frac{2(v_{i}+v_{j}+v_{k})}{n-3}\\
 & +2\bigl((a_{i}+a_{j})x_{ij}+(a_{i}+a_{k})x_{ik}+(a_{j}+a_{k})x_{jk}\bigr)
\end{align*}
and 
\begin{equation*}
q_{ijk}=\bigl(2(n-1)a_{i}+a_{j}+a_{k}\bigr)x_{jk}-\bigl(2(n-1)a_{j}+a_{i}+a_{k}\bigr)x_{ik}-\frac{v_{i}-v_{j}}{n}.  
\end{equation*}
The expression $p_{ijk}$ is symmetric in its three indices, whereas
$q_{jik}=-q_{ijk}$ and $q_{ijk}+q_{jki}+q_{kij}=0$. The deficit
in our inequality has the following sum-of-squares decomposition:
\begin{align*}
\frac{A_n}{4}-F-H {} &= \frac{n-3}{3(n-4)(n-2)^{2}}\sum_{i<j<k}p^{2}_{ijk}\\
 & +\frac{2n^{2}-9n+12}{6n(n-1)(n-2)^{2}}\sum_{i<j}\sum_{k\ne i,j}q^{2}_{ijk}\\
 & +\frac{(n-1)(n-2)(n-3)+2}{n(n-2)^{2}}M.
\end{align*}
%The identity follows by expanding the squares and using the zero row
%sums and the two normalizations; the detailed expansion is omitted.
To verify the decomposition, expand the squares and collect terms,
using the zero row sums and the two normalizations. This gives
\begin{align*}
\sum_{i<j<k}p_{ijk}^{2}
&=(n-4)^2-4(n-4)F-2(n-4)^2G\\
&\quad+\left((n-4)^2+\frac{2(n-4)}{n-3}\right)H,
\end{align*}
and
\begin{align*}
\sum_{i<j}\sum_{k\ne i,j}q_{ijk}^{2}
&=8(n-1)^2-2(n-1)F-4(n-1)^2G\\
&\quad-\left(4(n-1)^2+2-\frac{2}{n}\right)H.
\end{align*}
Substituting these expansions, together with
$M=2G-F-\frac{H}{n}$, into the right-hand side of the
sum-of-squares decomposition yields $\frac{A_n}{4}-F-H$.
All three coefficients are positive for $n>4$, since $2n^{2}-9n+12=(n-4)(2n-1)+8>0$.
Thus the identity proves the inequality, and equality forces $M=0$,
which is exactly \eqref{eq:equality-condition}.

Finally, rescaling $a$ and $X$ restores the original normalization
and proves both assertions. 
\end{proof}

\subsection{Proof of the symmetric-action estimate}
\begin{proof}[Proof of Proposition~\ref{prop:sharp-four-slot}]
The inequality is immediate when $S=0$. Otherwise, by homogeneity,
assume $|S|=1$ and choose an orthonormal eigenbasis in which 
\[
S=\operatorname{diag}(a_{1},\ldots,a_{n}),\quad\sum_{i}a_{i}=0,\quad\sum_{i}a^{2}_{i}=1.
\]
For an algebraic Weyl tensor $T$, write 
\[
\mathcal{F}_{S}(T)=|S\cdot T|^{2}+16|\mathring{T}(S)|^{2}.
\]
Let $\mathcal{B}$ be the family of summands in Lemma~\ref{lem:weyl-component-decomposition}.
The orthogonality assertions give the exact identities 
\begin{equation}
|W|^{2}=\sum_{T\in\mathcal{B}}|T|^{2},\quad\mathcal{F}_{S}(W)=\sum_{T\in\mathcal{B}}\mathcal{F}_{S}(T).\label{eq:action-block-additivity}
\end{equation}
It therefore suffices to bound $\mathcal{F}_{S}(W)$ on each of the
three types of summands.

\smallskip{}
\noindent \emph{Two distinct indices.} For $W^{[2]}$, let $x_{ij}=W_{ijij}$
for $i\ne j$ and $x_{ii}=0$, and put $E(X)=\sum_{i<j}x^{2}_{ij}$.
The matrix $X$ is symmetric and has zero row sums by Lemma~\ref{lem:weyl-component-decomposition}.
Each $x_{ij}$, $i<j$, occurs in four tensor components, up to sign,
and the symmetric-action multipliers are $2(a_{i}+a_{j})$. Consequently,
\begin{align*}
|W^{[2]}|^{2} & =4E(X),\\
|S\cdot W^{[2]}|^{2} & =16\sum_{i<j}(a_{i}+a_{j})^{2}x^{2}_{ij},\\
\mathring{W}^{[2]}(S) & =-\operatorname{diag}(Xa).
\end{align*}
All the hypotheses of Lemma~\ref{lem:sharp-edge-estimate} hold,
so 
\begin{align*}
\mathcal{F}_{S}(W^{[2]}) & =16\left(\sum_{i<j}(a_{i}+a_{j})^{2}x^{2}_{ij}+|Xa|^{2}\right)\\
 & \leq 4A_{n}E(X)=A_{n}|W^{[2]}|^{2}.
\end{align*}

\smallskip{}
\noindent \emph{Three distinct indices.} Fix $j<k$ and set $y_{i}=W_{ijik}$
for $i\notin\{j,k\}$. In this paragraph every sum over $i$ excludes
$j,k$. Each $y_{i}$ occurs in eight tensor components, up to sign,
and the symmetric-action multipliers are $2a_{i}+a_{j}+a_{k}$. The
contraction formula in Lemma~\ref{lem:weyl-component-decomposition}
thus yields 
\begin{align*}
|W^{[3]}_{jk}|^{2} & =8\sum_{i}y^{2}_{i},\\
\mathcal{F}_{S}(W^{[3]}_{jk}) & =8\sum_{i}(2a_{i}+a_{j}+a_{k})^{2}y^{2}_{i}+32\left(\sum_{i}a_{i}y_{i}\right)^{2}.
\end{align*}
The vector $a$ has unit length and is orthogonal to $\mathbf{1}=(1,\ldots,1)$.
Therefore 
\[
(2a_{i}+a_{j}+a_{k})^{2}\leq\left|2e_{i}+e_{j}+e_{k}-\frac{4}{n}\mathbf{1}\right|^{2}=6-\frac{16}{n}.
\]
Also, Cauchy-Schwarz gives $(\sum_{i}a_{i}y_{i})^{2}\leq(\sum_{i}a^{2}_{i})(\sum_{i}y^{2}_{i})\leq\sum_{i}y^{2}_{i}$.
It follows that 
\[
\mathcal{F}_{S}(W^{[3]}_{jk})\leq\left(10-\frac{16}{n}\right)|W^{[3]}_{jk}|^{2}.
\]
This coefficient is strictly smaller than $A_{n}$ as 
\begin{equation}
A_{n}-\left(10-\frac{16}{n}\right)=2-\frac{4}{n}+\frac{16}{n(n-2)^{2}}>0.\label{eq:three-index-strictness}
\end{equation}

\smallskip{}
\noindent \emph{Four distinct indices.} Fix $I=\{i,j,k,l\}$. The
contraction $\mathring{W}^{[4]}_{I}(S)$ is zero, and every component
of $S\cdot W^{[4]}_{I}$ has the same multiplier $a_{i}+a_{j}+a_{k}+a_{l}$.
Hence 
\begin{align*}
\mathcal{F}_{S}(W^{[4]}_{I}) & =(a_{i}+a_{j}+a_{k}+a_{l})^{2}|W^{[4]}_{I}|^{2}\\
 & \leq\left|e_{i}+e_{j}+e_{k}+e_{l}-\frac{4}{n}\mathbf{1}\right|^{2}|W^{[4]}_{I}|^{2}\\
 & =\left(4-\frac{16}{n}\right)|W^{[4]}_{I}|^{2}.
\end{align*}
The coefficient is again strictly smaller than $A_{n}$. When $n=4$,
it vanishes because $\sum_{i}a_{i}=0$.

Adding these three estimates using \eqref{eq:action-block-additivity}
proves \eqref{eq:intro-sharp}.

\smallskip{}
\noindent \emph{Optimality.} Let $W_{P}$ be the Weyl tensor of the
Einstein product 
\[
P=S^{2}(n-1)\times S^{n-2}\left(\frac{n-1}{n-3}\right),
\]
where $S^{d}(k)$ denotes the round sphere of sectional curvature
$k$. The Einstein constant of $P$ is $n-1$. Define 
\begin{equation}
S_{P}=\alpha g_{2}+\beta g_{n-2},\quad\alpha=\sqrt{\frac{n-2}{2n}},\quad\beta=-\sqrt{\frac{2}{n(n-2)}}.\label{eq:product-splitting}
\end{equation}
Thus $2\alpha+(n-2)\beta=0$ and $|S_{P}|=1$. In an adapted orthonormal
frame, the only independent nonzero components of $W_{P}$ are 
\begin{align}
(W_{P})_{1212} & =n-2,\\
(W_{P})_{iaia} &=-1, \text{ for } i\leq2<a,\nonumber \\
(W_{P})_{abab} & =\frac{2}{n-3}, \text{ for } 3\leq a<b\leq n.\label{eq:product-Weyl}
\end{align}
Let $\left(X_{P}\right)_{ij}$ be the corresponding matrix and $a=(\alpha,\alpha,\beta,\ldots,\beta)$.
We have 
\begin{align*}
E\left(X_{P}\right):=\sum_{i<j}(X_{P})^{2}_{ij} & =(n-2)^{2}+2(n-2)+\binom{n-2}{2}\frac{4}{(n-3)^{2}}\\
 & =\frac{(n-1)(n-2)^{2}}{n-3}.
\end{align*}
The first two entries of $X_{P}a$ equal $(n-2)(\alpha-\beta)=n\alpha$,
and each remaining entry equals $-2\alpha+2\beta=n\beta$. Thus $\mathring{W}_{P}(S_{P})=-nS_{P}$.
Similarly, 
\begin{align*}
F_{P}&:=\sum_{i<j}(a_{i}+a_{j})^{2}(X_{P})^{2}_{ij} \\
& =4\alpha^{2}(n-2)^{2}+2(n-2)(\alpha+\beta)^{2}+\binom{n-2}{2}\frac{16\beta^{2}}{(n-3)^{2}}\\
 & =\frac{2(n-2)^{3}+(n-4)^{2}+\frac{16}{n-3}}{n}.
\end{align*}
Substitution gives $F_{P}+n^{2}=\frac{A_{n}}{4}E(X_P)$. Since $|W_{P}|^{2}=4E(X_P)$
and $|S_{P}\cdot W_{P}|^{2}=16F_{P}$, we conclude that 
\[
\mathcal{F}_{S_{P}}(W_{P})=16(F_{P}+n^{2})=A_{n}|S_{P}|^{2}|W_{P}|^{2}.
\]
This proves optimality for every $n\geq4$, even when $S$ is an eigenvector. 
\end{proof}

\section{A sharp joint estimate}

\label{sec:joint-estimate}

To prove Proposition~\ref{prop:joint-estimate} we again use the
orthogonal decomposition from Lemma~\ref{lem:weyl-component-decomposition}.
On $W^{[2]}$ the commutator term vanishes, so the estimate follows
from Section~\ref{sec:sharp-estimates}. On the other two parts,
the symmetric and commutator actions depend on orthogonal components
of the eigenvalue vector of $S$. Keeping this dependence gives the
joint bound and explains why its two coefficients are attained by
different Weyl tensors.

\subsection{Proof of the joint estimate}
\begin{proof}[Proof of Proposition~\ref{prop:joint-estimate}]
The case $S=0$ is immediate. By homogeneity, we may assume that
$|S|=1$ and choose an orthonormal basis $e_{1},...,e_{n}$ with 
\[
S=\operatorname{diag}(a_{1},\ldots,a_{n}),\quad\sum_{i}a_{i}=0,\quad\sum_{i}a^{2}_{i}=1.
\]
For an algebraic Weyl tensor $T$, put 
\begin{align*}
\mathcal{F}_{S}(T) & =|S\cdot T|^{2}+16|\mathring{T}(S)|^{2},\\
\mathcal{J}_{S}(T) & =\frac{3}{8}|[\ad_{S},T]|^{2}+\frac{\mathcal{F}_{S}(T)}{A_{n}}.
\end{align*}
We must show $\mathcal{J}_{S}(W)\leq|W|^{2}$. Let $\mathcal{B}$
denote the family of summands in \eqref{eq:weyl-component-decomposition}.
In the chosen basis, 
\[
[\ad_{S},T]_{ijkl}=(a_{k}+a_{l}-a_{i}-a_{j})T_{ijkl},
\]
so the commutator preserves the disjoint component supports of these
summands. Together with the orthogonality assertions of Lemma~\ref{lem:weyl-component-decomposition},
this gives 
\begin{equation}
\mathcal{J}_{S}(W)=\sum_{T\in\mathcal{B}}\mathcal{J}_{S}(T),\quad|W|^{2}=\sum_{T\in\mathcal{B}}|T|^{2}.\label{eq:joint-block-additivity}
\end{equation}
We therefore prove the estimate separately on each type of summand.

\smallskip{}
\noindent \emph{Two distinct indices.} For $T=W^{[2]}$, the unordered
index pairs of every nonzero component coincide. Hence $[\ad_{S},T]=0$,
and Proposition~\ref{prop:sharp-four-slot} gives 
\[
\mathcal{J}_{S}(T)=\frac{\mathcal{F}_{S}(T)}{A_{n}}\leq|T|^{2}.
\]

\smallskip{}
\noindent \emph{Three distinct indices.} Fix $j<k$ and write $T=W^{[3]}_{jk}$
and $y_{i}=W_{ijik}$ for $i\notin\{j,k\}$. Sums involving $y_{i}$
exclude $j,k$. Set 
\[
d=a_{j}-a_{k},\quad a'=a-\frac{1}{2}d(e_{j}-e_{k}),\quad S'=\operatorname{diag}(a'_{1},\ldots,a'_{n}).
\]
Thus $a'_{j}=a'_{k}=\frac{a_{j}+a_{k}}{2}$, while $a'_{i}=a_{i}$
for every other index. It follows that 
\[
\sum^{n}_{r=1}a'_{r}=0,\quad|S'|^{2}=1-\frac{d^{2}}{2},\quad0\leq d^{2}\leq2.
\]
The symmetric-action multipliers are $2a_{i}+a_{j}+a_{k}$ and $-\sum_{i}a_{i}y_{i}$
is unchanged by this change. Consequently $\mathcal{F}_{S}(T)=\mathcal{F}_{S'}(T)$.
On the other hand, the commutator multipliers are $\pm d$. Thus 
\[
|[\ad_{S},T]|^{2}=d^{2}|T|^{2}.
\]
The three-index estimate proved in Section~\ref{sec:sharp-estimates},
applied to $S'$ and rescaled by homogeneity, gives 
\[
\mathcal{F}_{S}(T)\leq\left(10-\frac{16}{n}\right)\left(1-\frac{d^{2}}{2}\right)|T|^{2}.
\]
For $T\ne0$, we obtain 
\begin{align*}
\frac{\mathcal{J}_{S}(T)}{|T|^{2}} & \leq\frac{d^{2}}{2}\,\frac{3}{4}+\left(1-\frac{d^{2}}{2}\right)\frac{10-\frac{16}{n}}{A_{n}}\\
 & \leq\max\left\{ \frac{3}{4},\frac{10-\frac{16}{n}}{A_{n}}\right\} <1.
\end{align*}
The last strict inequality follows from \eqref{eq:three-index-strictness}.

\smallskip{}
\noindent \emph{Four distinct indices.} Relabel the four indices
of a summand $T=W^{[4]}_{I}$ using $I=\{1,2,3,4\}$ and set 
\[
u_{1}=T_{1234},\quad u_{2}=T_{1342},\quad u_{3}=T_{1423}.
\]
The curvature symmetries determine all the other components from these
three numbers, and the first Bianchi identity is exactly $u_{1}+u_{2}+u_{3}=0$.
Each $u_{\alpha}$ occurs eight times, up to sign; hence, with $U=u^{2}_{1}+u^{2}_{2}+u^{2}_{3}$,
\[
|T|^{2}=8U.
\]
Separate the mean of the four eigenvalues by writing 
\begin{align*}
c & =\frac{a_{1}+a_{2}+a_{3}+a_{4}}{4},\\
z & =(a_{1}-c,a_{2}-c,a_{3}-c,a_{4}-c),\quad\sum^{4}_{r=1}z_{r}=0.
\end{align*}
The three commutator multipliers are 
\begin{align*}
d_{1} & =a_{3}+a_{4}-a_{1}-a_{2},\\
d_{2} & =a_{4}+a_{2}-a_{1}-a_{3},\\
d_{3} & =a_{2}+a_{3}-a_{1}-a_{4}.
\end{align*}
They depend only on $z$, and direct expansion gives 
\begin{align*}
\sum^{3}_{\alpha=1}d^{2}_{\alpha} & =4\sum^{4}_{r=1}a^{2}_{r}-\left(\sum^{4}_{r=1}a_{r}\right)^{2}=4|z|^{2},\\
|[\ad_{S},T]|^{2} & =8\sum^{3}_{\alpha=1}d^{2}_{\alpha}u^{2}_{\alpha}.
\end{align*}
The first Bianchi identity implies the identity 
\[
2U-3u^{2}_{\alpha}=(u_{\beta}-u_{\gamma})^{2},
\]
where $(\alpha,\beta,\gamma)$ runs cyclically through $(1,2,3)$.
Multiplying by $d^{2}_{\alpha}$ and summing yields 
\begin{align*}
\frac{3}{8}|[\ad_{S},T]|^{2} & =|z|^{2}|T|^{2}-\sum_{\mathrm{cyc}}d^{2}_{\alpha}(u_{\beta}-u_{\gamma})^{2}\\
 & \leq|z|^{2}|T|^{2}.
\end{align*}
Meanwhile, $\mathring{T}(S)=0$ and $S\cdot T=4cT$. Therefore 
\begin{equation}
\mathcal{J}_{S}(T)\leq\left(|z|^{2}+\frac{16}{A_{n}}c^{2}\right)|T|^{2}.\label{eq:four-index-joint-bound}
\end{equation}
If $n=4$, we have $\tr S=c=0$ and $|z|^{2}=1$. If $n>4$, the remaining
eigenvalues have sum $-4c$, so Cauchy-Schwarz gives 
\begin{align}
1 & =|z|^{2}+4c^{2}+\sum^{n}_{r=5}a^{2}_{r}\nonumber \\
 & \geq|z|^{2}+\left(4+\frac{16}{n-4}\right)c^{2}=|z|^{2}+\frac{4n}{n-4}c^{2}.\label{eq:four-index-trace-bound}
\end{align}
Since $A_{n}\geq8$, 
\[
\frac{16}{A_{n}}\leq2<\frac{4n}{n-4}.
\]
Equations~\eqref{eq:four-index-joint-bound} and \eqref{eq:four-index-trace-bound}
prove $\mathcal{J}_{S}(T)\leq|T|^{2}$.

Adding the estimates with \eqref{eq:joint-block-additivity} proves
\eqref{eq:intro-joint} for $|S|=1$. Homogeneity gives the statement
for arbitrary $S$.

\smallskip{}
\noindent \emph{Optimality of both coefficients.} The example $(W_{P},S_{P})$
in \eqref{eq:product-splitting}-\eqref{eq:product-Weyl} only has
$W^{[2]}$ components in an eigenbasis of $S_{P}$. Thus $[\ad_{S_{P}},W_{P}]=0$,
while $\mathcal{F}_{S_{P}}(W_{P})=A_{n}|S_{P}|^{2}|W_{P}|^{2}$ by
Section~\ref{sec:sharp-estimates}. This proves optimality of $\frac{1}{A_{n}}$.

For the other coefficient, use 
\[
S=\operatorname{diag}\left(\frac{1}{2},\frac{1}{2},-\frac{1}{2},-\frac{1}{2},0,\ldots,0\right)
\]
and let $W$ have only the four-index summand on $\{1,2,3,4\}$, with
$(u_{1},u_{2},u_{3})=(2,-1,-1)$. The relation $u_{1}+u_{2}+u_{3}=0$
ensures Bianchi, and every Ricci contraction vanishes because the
four indices are distinct. Hence $W$ is an algebraic Weyl tensor.
Here $|S|=1$, $c=0$, and $(d_{1},d_{2},d_{3})=(-2,0,0)$, so 
\begin{align*}
|W|^{2} & =8(4+1+1)=48,\\
|[\ad_{S},W]|^{2} & =8(4\cdot4)=128,\\
S\cdot W & =0,\quad\mathring{W}(S)=0.
\end{align*}
Thus $(\frac{3}{8})|[\ad_{S},W]|^{2}=|S|^{2}|W|^{2}$, proving optimality
of $\frac{3}{8}$. Both examples are eigentensor pairs: their eigenvalues
are $-n$ and $0$, respectively. 
\end{proof}

\section{Proof of the main rigidity theorem}

\label{sec:main-proof}

The joint estimate combines the two representations of $Q$ into a
single new nonnegativity condition. On a complete Einstein manifold,
this bound first gives local symmetry. Comparing with Theorem~\ref{thm:symmetric-classification}
then excludes every nonround symmetric model. The classification used
in the final step is established in Section~\ref{sec:symmetric-classification}
from the eigenvalue calculations in Appendix~\ref{app:symmetric-spectra}.

Throughout this section $n\geq5$, and $A_{n},B_{n},C_{n}$ are the
constants in \eqref{eq:intro-constants}. The bounds $8<A_{n}<12$
and $0<B_{n}<1$ show that 
\[
\frac{4}{3}<\frac{A_{n}+8}{A_{n}+B_{n}}<2,
\]
and hence 
\[
n-1<C_{n}<\frac{3(n-1)}{2}<N:=\frac{(n-1)(n+2)}{2}.
\]

\subsection{A pointwise curvature inequality}

\begin{proposition}\label{prop:pointwise-main} Let $W$ be an algebraic
Weyl tensor on an $n$-dimensional Euclidean space, $n\geq5$, and
set $R=\kappa I+W$ for $\kappa\in\mathbb{R}$. Then 
\begin{equation}
Q(\kappa I+W,W)\geq\frac{4(A_{n}+B_{n})}{3(A_{n}+8)}\Sigma_{C_{n}}(\mathring{R})|W|^{2}.\label{eq:pointwise-main}
\end{equation}
\end{proposition}
\begin{proof}
The assertion is immediate when $W=0$, so assume $W\ne0$. Choose
an ordered orthonormal eigenbasis $S_{1},\ldots,S_{N}$ of $\mathring{R}$,
with eigenvalues $\lambda_{i}$. Since $\mathring{R}=\kappa\Id+\mathring{W}$,
we also have 
\[
\mathring{W}(S_{i})=\mu_{i}S_{i},\quad\mu_{i}=\lambda_{i}-\kappa.
\]
Adding $B_{n}|W|^{2}$ to the estimate \eqref{eq:intro-joint} gives
the weights 
\[
w_{i}=\frac{3A_{n}}{8}|[\ad_{S_{i}},W]|^{2}+|S_{i}\cdot W|^{2}+16\mu^{2}_{i}+B_{n}|W|^{2},
\]
with 
\[
0\leq w_{i}\leq(A_{n}+B_{n})|W|^{2}.
\]
The two identities in Proposition~\ref{prop:two-representations}
give 
\begin{align*}
\sum_{i}\lambda_{i}w_{i} & =\frac{3(A_{n}+8)}{4}Q(\kappa I+W,W),\\
\sum_{i}w_{i} & =\frac{3(n-1)(A_{n}+8)}{4}|W|^{2}=C_{n}(A_{n}+B_{n})|W|^{2}.
\end{align*}
Thus $C_{n}$ is the total weight divided by the common upper bound
for one weight. Lemma~\ref{lem:weight-comparison} yields 
\[
\frac{3(A_{n}+8)}{4}Q(\kappa I+W,W)\geq(A_{n}+B_{n})\Sigma_{C_{n}}(\mathring{R})|W|^{2},
\]
which proves the proposition. 
\end{proof}

\subsection{Local symmetry and eigenvalue comparison}

Combining \eqref{eq:pointwise-main} with the Bochner identity \eqref{eq:bochner}
gives 
\[
\frac{1}{2}\Delta|W|^{2}\geq|\nabla W|^{2}+\frac{8(A_{n}+B_{n})}{3(A_{n}+8)}\Sigma_{C_{n}}(\mathring{R})|W|^{2}.
\]
The nonnegativity assumption therefore forces $W$ to be parallel
on a closed manifold. The remaining step is geometric: the same hypothesis
shows that the metric is in the open part of the cone for a shifted
operator, whose only locally symmetric models are the round spheres.
\begin{proof}[Proof of Theorem~\ref{thm:main}]
Write $\Ric=(n-1)\kappa g$. Monotonicity of the eigenvalue averages
gives 
\[
0\leq\frac{\Sigma_{C_{n}}(\mathring{R})}{C_{n}}\leq\frac{\tr\mathring{R}}{N}=\kappa.
\]
If $\kappa=0$, the zero-trace observation following \eqref{eq:kyfan},
together with $C_{n}<N$, forces $\mathring{R}=0$ and hence the metric
is flat. 

Suppose $\kappa>0$. Completeness and Bonnet-Myers imply that $M$
is compact. Integrating \eqref{eq:bochner} and using Proposition~\ref{prop:pointwise-main},
we obtain 
\[
0=\int_{M}\bigl(|\nabla W|^{2}+2Q(\kappa I+W,W)\bigr)\,dV_{g}.
\]
Both summands are nonnegative, so $\nabla W=0$. Since $\kappa$ is
constant, $\nabla R=0$, and $g$ is locally symmetric.

Set $k=\frac{n+2}{2}$, so that $k<n-1<C_{n}<N$. The function $l\mapsto\Sigma_{l}(\mathring{R})$
is convex and piecewise affine: its successive slopes are the ordered
eigenvalues. Applying convexity between $k$ and $N$ yields 
\begin{align*}
0\leq\Sigma_{C_{n}}(\mathring{R}) & \leq\frac{N-C_{n}}{N-k}\Sigma_{k}(\mathring{R})+\frac{C_{n}-k}{N-k}\tr\mathring{R},\\
\Sigma_{k}(\mathring{R}) & \geq-\frac{N(C_{n}-k)}{N-C_{n}}\kappa.
\end{align*}
This argument applies without change to fractional indices. Because
\[
\mathring{W}+\beta_{n}\Id=\mathring{R}+\frac{2(n-1)\kappa}{n+2}\Id,
\]
it follows that 
\[
\Sigma_{k}(\mathring{W}+\beta_{n}\Id)\geq\kappa\left(n-1-\frac{N(C_{n}-k)}{N-C_{n}}\right).
\]
The quantity in parentheses is positive precisely when $C_{n}<\frac{3n(n+2)}{2(n+4)}$.
This strict inequality follows from 
\[
\frac{3n(n+2)}{2(n+4)}-C_{n}=\frac{3(n^{2}+n+4)(A_{n}-8)}{4(n+4)(A_{n}+B_{n})}>0,
\]
since $n\geq5$ and $A_{n}>8$. Thus $\mathring{W}+\beta_{n}\Id$
is $k$-positive.

The universal cover is complete, simply connected, locally symmetric,
and Einstein. Theorem~\ref{thm:symmetric-classification} now identifies
it with a round sphere, up to scaling, as all the nonround models
in that theorem have vanishing $k$-sum for the shifted operator.
Hence $g$ has constant positive sectional curvature. Together with
the flat case, this proves the theorem. 
\end{proof}

\section{Böhm-Wilking shifts and their consequences}

\label{sec:shifted}

The maps $\ell_{a,b}$ that create shifted curvature operators are
explored within the framework of Bochner formulas. We first determine
the map that absorbs the Ricci and scalar terms in the Bochner formula
for forms that involves the curvature operator of the second kind.
We then prove the harmonic-form and Ricci estimates and the shifted
Einstein rigidity theorem. The classification of the locally symmetric
models is proved in Section~\ref{sec:symmetric-classification}.

\subsection{The choice of shift}

For $n\geq3$, the Böhm-Wilking map in our curvature convention is
\[
\ell_{a,b}(R)=R+bE\owedge g+\frac{as}{n}g\owedge g,
\]
where $E=\Ric_{R}-\frac{s}{n} g$ is the traceless Ricci tensor. 
It leaves the Weyl part unchanged and multiplies the traceless Ricci
and scalar parts by $1+(n-2)b$ and $1+2(n-1)a$, respectively; compare
\cite[Theorem~2]{BohmWilking}. By definition, $\mathring{R}_{a,b}$
is the curvature operator of the second kind induced by $\ell_{a,b}(R)$.
The contraction identity \eqref{eq:ricci-correction-conventions}
gives 
\[
\mathring{R}_{a,b}=\mathring{R}+bJ_{E}+\frac{2as}{n}\Id.
\]
For the particular parameters in \eqref{eq:intro-shift-parameters}, we have
\begin{align*}
 & \mathring{R}_{\gamma_{n}}=\mathring{R}+\frac{2}{n+4}J_{E}+\frac{2s}{n(n+2)}\Id,\\
 & \tr\mathring{R}_{\gamma_{n}}=\frac{3}{2}s.
\end{align*}
Here $\tr J_{E}=0$, so the trace formula follows from \eqref{eq:second-kind-trace}.
If $\Ric=(n-1)\kappa g$, then $E=0$ and 
\begin{align*}
  \mathring{R}_{\gamma_{n}} &=\mathring{R}+\frac{2(n-1)\kappa}{n+2}\Id=\mathring{W}+\beta_{n}\Id,\\
  \beta_{n} &=\frac{3n\kappa}{(n+2)}.
\end{align*}
The following calculation explains the choice of both parameters:
it cancels the two correction terms simultaneously in the Bochner
formulas for all forms.

\subsection{An exact Bochner identity for differential forms}

Let $Q_{p}(R)$ denote the curvature endomorphism in the Weitzenböck
formula on $p$-forms, 
\[
\Delta_{H}=\nabla^{*}\nabla+Q_{p}(R),
\]
where $\Delta_{H}=d\delta+\delta d$ is nonnegative. For a symmetric
endomorphism $H$, extend the notation $\ad_{H}$ to forms by the
action 
\[
(\ad_{H}\eta)(X_{1},\ldots,X_{p})=\sum^{p}_{j=1}\eta(X_{1},\ldots,HX_{j},\ldots,X_{p}).
\]
Recall that $\gamma_{n}=\frac{2}{n+4}$ and $\mathring{R}_{\gamma_{n}}=\mathring{R}_{\frac{1}{n+2},\gamma_{n}}$.

\begin{proposition}\label{prop:forms-exact-deformation} Let $R$
be an algebraic curvature tensor in dimension $n\geq3$. If $S_{1},\ldots,S_{N}$
is an orthonormal eigenbasis of $\mathring{R}_{\gamma_{n}}$, with
eigenvalues $\lambda_{1},\ldots,\lambda_{N}$, then every $p$-form
$\eta$ satisfies 
\begin{equation}
\frac{3}{2}\langle Q_{p}(R)\eta,\eta\rangle=\frac{3}{2}Q_{p}(R,\eta)=\sum^{N}_{i=1}\lambda_{i}|\ad_{S_{i}}\eta|^{2}.\label{eq:forms-exact-deformation}
\end{equation}
Consequently, on a Riemannian manifold a harmonic $p$-form satisfies
\[
\frac{1}{2}\Delta|\eta|^{2}=|\nabla\eta|^{2}+\frac{2}{3}\sum^{N}_{i=1}\lambda_{i}|\ad_{S_{i}}\eta|^{2}.
\]
\end{proposition}
\begin{proof}
Write $s=\tr\Ric_{R}$ and $E=\Ric_{R}-\frac{s}{n}g$. The second-kind
Bochner formula of \cite[Proposition~2.1]{NPW}, applied to any algebraic
curvature tensor $T$, reads 
\begin{equation}
\frac{3}{2}Q_{p}(T,\eta)=\sum_{i}\nu_{i}|\ad_{U_{i}}\eta|^{2}
  +\frac{n-2p}{n}\langle\ad_{\Ric_{T}}\eta,\eta\rangle+\frac{p^{2}s_{T}}{n^{2}}|\eta|^{2},\label{eq:forms-NPW}
\end{equation}
where $\mathring{T}(U_{i})=\nu_{i}U_{i}$ and the $U_{i}$ are orthonormal.
The point is to apply this formula to $T=\ell_{a,b}(R)$, and then
choose $a,b$ so that the two correction terms are absorbed.

First, for any symmetric $H$, the Kulkarni-Nomizu convention gives
\begin{equation}
Q_{p}(H\owedge g)=(n-2p)\ad_{H}+p(\tr H)\Id.\label{eq:forms-KN-curvature}
\end{equation}
Indeed, diagonalize $H$ with eigenvalues $h_{1},\ldots,h_{n}$. On
the basis form $e^{i_{1}}\wedge\cdots\wedge e^{i_{p}}$, indexed by
$I=\{i_{1},\ldots,i_{p}\}$, the curvature endomorphism has eigenvalue
\[
\sum_{i\in I,\,j\notin I}(h_{i}+h_{j})=(n-p)\sum_{i\in I}h_{i}+p\sum_{j\notin I}h_{j}=(n-2p)\sum_{i\in I}h_{i}+p\sum^{n}_{j=1}h_{j}.
\]
In particular, $Q_{p}(I)=p(n-p)\Id$.

Now use 
\[
T=\ell_{a,b}(R)=R+bE\owedge g+\frac{2as}{n}I.
\]
Linearity, \eqref{eq:forms-KN-curvature}, and the Ricci contractions
in Section~\ref{sec:preliminaries} give 
\begin{align*}
Q_{p}(T) & =Q_{p}(R)+b(n-2p)\ad_{E}+\frac{2as}{n}p(n-p)\Id,\\
\Ric_{T} & =\bigl(1+(n-2)b\bigr)E+\bigl(1+2(n-1)a\bigr)\frac{s}{n}g,\\
s_{T} & =\bigl(1+2(n-1)a\bigr)s.
\end{align*}
Substitute these expressions into \eqref{eq:forms-NPW}, using $\ad_{g}=p\Id$
on $p$-forms, and move the last two terms to the left. Collecting
the traceless Ricci and scalar contributions yields 
\begin{align}
\sum_{i}\nu_{i}|\ad_{U_{i}}\eta|^{2}={} & \frac{3}{2}Q_{p}(R,\eta)
  +\frac{n-2p}{n}\left(\frac{n+4}{2}b-1\right)\langle\ad_{E}\eta,\eta\rangle\nonumber \\
 & +\frac{p(n-p)}{n^{2}}\bigl((n+2)a-1\bigr)s|\eta|^{2}.\label{eq:forms-general-deformation}
\end{align}
For $a=\frac{1}{n+2}$ and $b=\frac{2}{n+4}$, both coefficients vanish,
independently of $p$. This proves \eqref{eq:forms-exact-deformation}.
The harmonic-form identity follows from the Weitzenböck formula with
our stated Laplacian convention. 
\end{proof}

\subsection{Harmonic forms and Ricci curvature}
\begin{proof}[Proof of Theorem~\ref{thm:shifted-harmonic}]
For $1\leq p\leq n-1$, the symmetric action has the individual and
total bounds 
\begin{align}
|\ad_{S}\eta|^{2} & \leq\frac{p(n-p)}{n}|S|^{2}|\eta|^{2},\quad S\in S^{2}_{0}(V),\nonumber \\
\sum^{N}_{i=1}|\ad_{S_{i}}\eta|^{2} & =\frac{p(n-p)(n+2)}{2n}|\eta|^{2}.\label{eq:forms-action-weights}
\end{align}
These are proven in \cite[Lemma~3.7]{NPW}. We briefly explain how
the constants arise. For $|S|=1$, diagonalize $S$ with eigenvalues
$a_{i}$. Since $\sum_{i}a_{i}=0$, every $p$-element subset $I$
satisfies 
\[
\left(\sum_{i\in I}a_{i}\right)^{2}=\left\langle a,\mathbf{1}_{I}-\frac{p}{n}\mathbf{1}\right\rangle ^{2}\leq\left|\mathbf{1}_{I}-\frac{p}{n}\mathbf{1}\right|^{2}=\frac{p(n-p)}{n}.
\]
Expanding $\eta$ in the corresponding basis proves the first bound.
For the second, apply \eqref{eq:forms-exact-deformation} to $R=I$:
since $Q_{p}(I)=p(n-p)\Id$ and 
\[
\mathring{I}_{\gamma_{n}}=\frac{3n}{n+2}\Id,
\]
the assertion follows.

Take an ordered orthonormal eigenbasis of $\mathring{R}_{\gamma_{n}}$.
For $\eta\ne0$, the weights $|\ad_{S_{i}}\eta|^{2}$ have upper bound
$\frac{p(n-p)}{n}|\eta|^{2}$, and their total is $\frac{n+2}{2}$
times this bound. Lemma~\ref{lem:weight-comparison} and \eqref{eq:forms-exact-deformation}
therefore give 
\begin{equation}
Q_{p}(R,\eta)\geq\frac{2p(n-p)}{3n}\Sigma_{\frac{n+2}{2}}\bigl(\mathring{R}_{\gamma_{n}}\bigr)|\eta|^{2}.\label{eq:forms-spectral-bound}
\end{equation}
The same inequality is immediate when $\eta=0$.

Under the theorem's nonnegativity assumption, $Q_{p}(R)$ is nonnegative.
If $\eta$ is harmonic, integrating the Bochner formula for $\eta$
on the closed manifold gives 
\[
0=\int_{M}\bigl(|\nabla\eta|^{2}+Q_{p}(R,\eta)\bigr)\,dV_{g}.
\]
Both terms are nonnegative, so $\eta$ is parallel. If the spectral
sum is positive at one point, \eqref{eq:forms-spectral-bound} forces
this parallel form to vanish there, hence everywhere. Harmonic functions
and top-degree forms are parallel as well. The cohomology statements
follow from Hodge theory and, in the oriented case, Poincaré duality. 
\end{proof}

\begin{proposition}\label{prop:shifted-ricci} Every Riemannian manifold
of dimension $n\geq3$ satisfies 
\begin{equation}
\Ric\geq\frac{2(n-1)}{3n}\Sigma_{\frac{n+2}{2}}\bigl(\mathring{R}_{\gamma_{n}}\bigr)g.\label{eq:forms-ricci-bound}
\end{equation}
If the eigenvalue sum is positive everywhere and the manifold is closed
and connected, then the fundamental group is finite and the universal
covering is a rational homology sphere. Under this strict condition,
a closed, connected manifold of dimension three or four is diffeomorphic
to a spherical space form. \end{proposition}
\begin{proof}
Since $Q_{1}(R)=\Ric$, inequality \eqref{eq:forms-ricci-bound} is
the case $p=1$ of \eqref{eq:forms-spectral-bound}; no sign assumption
is required. Under strict positivity, compactness gives a positive
uniform Ricci lower bound. Bonnet-Myers implies that the fundamental
group is finite and the universal cover is compact. The lifted metric
on a covering manifold satisfies the same strict spectral condition.
The universal cover is orientable, so Theorem~\ref{thm:shifted-harmonic}
makes it a rational homology sphere.

In dimension three, Hamilton's theorem for positive Ricci curvature
\cite{Hamilton1982} gives the differentiable conclusion. In dimension
four, \eqref{eq:forms-spectral-bound} also makes $Q_{2}(R)$ positive
definite. For any orthonormal four-frame, the unit two-form 
$$\eta=\frac{e^{1}\wedge e^{2}+e^{3}\wedge e^{4}}{\sqrt{2}}$$
satisfies 
\[
Q_{2}(R,\eta)=R_{1313}+R_{1414}+R_{2323}+R_{2424}-2R_{1234}.
\]
Thus the metric has positive isotropic curvature. Together with the
fact that the fundamental group is finite, the four-dimensional classification
shows that it is a spherical space form; see \cite[proof of Theorem~1.5]{Li2025}. 
\end{proof}

\subsection{The middle degree}

When $n=2p$, the traceless Ricci correction in \eqref{eq:forms-general-deformation}
vanishes for every $b$. Taking $a=\frac{1}{n+2}$ and $b=-\frac{1}{n-2}$
removes the traceless Ricci part of $R$ itself and gives the following
additional consequence.

\begin{corollary}\label{cor:forms-middle-degree} Suppose $n=2p\geq4$.
If $U_{i}$ is an orthonormal eigenbasis of 
\[
\mathring{W}+\frac{3s}{(n-1)(n+2)}\Id
\]
with eigenvalues $\nu_{i}$, then every $p$-form satisfies 
\[
\frac{3}{2}Q_{p}(R,\eta)=\sum_{i}\nu_{i}|\ad_{U_{i}}\eta|^{2}
\]
and 
\[
Q_{\frac{n}{2}}(R,\eta)\geq\frac{n}{6}\Sigma_{\frac{n+2}{2}}\left(\mathring{W}+\frac{3s}{(n-1)(n+2)}\Id\right)|\eta|^{2}.
\]
On a closed, connected manifold, nonnegativity of this eigenvalue
sum makes every harmonic middle-degree form parallel. If, in addition,
the sum is positive at one point, then $H^{\frac{n}{2}}(M;\mathbb{R})=0$.
\end{corollary}
\begin{proof}
The Weyl decomposition gives 
\[
\ell_{\frac{1}{n+2},-\frac{1}{n-2}}(R)=W+\frac{3s}{(n-1)(n+2)}I.
\]
The exact identity follows from \eqref{eq:forms-general-deformation}.
Applying the weight bounds \eqref{eq:forms-action-weights} yields
the stated estimate, since $\frac{2p(n-p)}{3n}=\frac{n}{6}$. The
conclusions for harmonic forms follow by the same integrated Bochner
argument used above. 
\end{proof}

\subsection{A shifted Einstein rigidity criterion}

\label{subsec:shifted-einstein-rigidity}

We now prove Theorem~\ref{thm:shifted-improved}. The shift in that
theorem is adapted to the symmetric-action representation of the curvature
term. Throughout this subsection, $n\geq4$ and 
\[
\alpha_{n}=\frac{2(n-1)(n+8)}{nA_{n}}.
\]
Since $8\leq A_{n}<12$, we have 
\[
1<\alpha_{n}\leq\frac{(n-1)(n+8)}{4n}<\frac{n+2}{2}<N:=\frac{(n-1)(n+2)}{2}.
\]
The strict middle inequality is equivalent to $n^{2}-3n+8>0$.

The comparison of indices stated in the introduction follows from
\[
\frac{8(n-2)(2n-1)}{n(n-1)}-A_{n}=\frac{4((n-2)^{3}-n)}{(n-1)(n-2)^{2}}>0.
\]

\begin{proposition}\label{prop:shifted-einstein-Q} Let $W$ be an
algebraic Weyl tensor on an $n$-dimensional Euclidean space, $n\geq4$,
and let $\kappa\in\mathbb{R}$. Then 
\begin{equation}
Q(\kappa I+W,W)\geq\frac{A_{n}}{6}\Sigma_{\alpha_{n}}\left(\mathring{W}+\frac{3n\kappa}{n+8}\Id\right)|W|^{2}.\label{eq:shifted-einstein-Q-bound}
\end{equation}
\end{proposition}
\begin{proof}
The assertion is immediate when $W=0$, so assume $W\ne0$. Choose
an orthonormal eigenbasis $S_{1},\ldots,S_{N}$ of $\mathring{W}$,
with increasingly ordered eigenvalues $\mu_{i}$, and put 
\[
\xi_{i}=\mu_{i}+\frac{3n\kappa}{n+8},\quad w_{i}=|S_{i}\cdot W|^{2}+16\mu^{2}_{i}.
\]
The contraction identities \eqref{eq:weyl-traces}-\eqref{eq:slot-cubic}
give 
\begin{align*}
\sum_{i}w_{i} & =\frac{2(n-1)(n+8)}{n}|W|^{2},\\
\sum_{i}\mu_{i}w_{i} & =-36\tr(W^{3})+12\tr((\mathring{W})^{3})=6Q(W,W).
\end{align*}
Consequently, by \eqref{eq:Q-polynomial}, 
\[
\sum_{i}\xi_{i}w_{i}=6Q(\kappa I+W,W),\quad\sum_{i}w_{i}=\alpha_{n}A_{n}|W|^{2}.
\]
Proposition~\ref{prop:sharp-four-slot} gives $0\leq w_{i}\leq A_{n}|W|^{2}$.
Applying Lemma~\ref{lem:weight-comparison} therefore proves \eqref{eq:shifted-einstein-Q-bound}. 
\end{proof}

\begin{proof}[Proof of Theorem~\ref{thm:shifted-improved}]
Write $\Ric=(n-1)\kappa g$, where $\kappa$ is constant, and put
\[
T=\mathring{W}+\frac{3n\kappa}{n+8}\Id.
\]
This is the operator in the hypothesis. Since $\tr\mathring{W}=0$,
monotonicity of eigenvalue averages gives 
\[
0\leq\frac{\Sigma_{\alpha_{n}}(T)}{\alpha_{n}}\leq\frac{\tr T}{N}=\frac{3n\kappa}{n+8}.
\]
Thus $\kappa\geq0$. If $\kappa=0$, then $\tr T=0$ and $\alpha_{n}<N$
force $T=\mathring{W}=0$ by the zero-trace observation following
\eqref{eq:kyfan}. Equation~\eqref{eq:weyl-traces} gives $W=0$,
so the metric is flat.

Suppose $\kappa>0$. Bonnet-Myers gives compactness, and Proposition~\ref{prop:shifted-einstein-Q}
makes $Q(\kappa I+W,W)$ nonnegative. Integrating \eqref{eq:bochner}
therefore gives 
\[
0=\int_{M}\bigl(|\nabla W|^{2}+2Q(\kappa I+W,W)\bigr)\,dV_{g}.
\]
It follows that $\nabla W=0$, and hence $\nabla R=0$.

For $n\geq5$, set $k=\frac{n+2}{2}$. The bound $\alpha_{n}<k$ and
monotonicity of eigenvalue averages imply $\Sigma_{k}(T)\geq0$. Moreover,
\[
\mathring{W}+\beta_{n}\Id=T+\frac{18n\kappa}{(n+2)(n+8)}\Id.
\]
The added scalar is positive, so $\Sigma_{k}(\mathring{W}+\beta_{n}\Id)>0$.
Applying Theorem~\ref{thm:symmetric-classification} to the universal
cover now gives constant positive sectional curvature.

For $n=4$, we have $\alpha_{4}=\frac{9}{4}$ and $T=\mathring{R}$.
The compact simply connected locally symmetric Einstein models in
this dimension are $S^{4}$, $\mathbb{CP}^{2}$, and $S^{2}\times S^{2}$.
Normalize $\Ric=\frac{1}{2}g$. The eigenvalues in Appendix~\ref{app:symmetric-spectra}
give, respectively for the two nonround models, 
\begin{align*}
\Sigma_{\frac{9}{4}}(\mathring{R}_{\mathbb{CP}^{2}}) & =\frac{9}{4}\left(-\frac{1}{6}\right)=-\frac{3}{8},\\
\Sigma_{\frac{9}{4}}(\mathring{R}_{S^{2}\times S^{2}}) & =-\frac{1}{2}.
\end{align*}
Both violate the hypothesis, so the universal cover is round in this
case as well. Together with the flat case, this proves the theorem. 
\end{proof}

\section{The locally symmetric models}

\label{sec:symmetric-classification}

We prove Theorem~\ref{thm:symmetric-classification}, completing
the spectral classification used in the two rigidity arguments above.
The proof uses only the COSK eigenvalues of symmetric spaces and the
known rigidity theorem of Dai-Fu \cite{DaiFu}. The eigenvalues and
their multiplicity calculations can be found in Appendix~\ref{app:symmetric-spectra}.
We use them here through a simple test involving just one or two negative
eigenspaces.
\begin{proof}[Proof of Theorem~\ref{thm:symmetric-classification}]
Put $k=\frac{n+2}{2}$ and 
\[
F=\Sigma_{k}(\mathring{W}+\beta_{n}\Id)=\Sigma_{k}\left(\mathring{R}+\frac{2(n-1)\kappa}{n+2}\Id\right).
\]
The trace of the shifted operator is $\frac{3n(n-1)\kappa}{2}$. Since
the average of the lowest $k$ eigenvalues does not exceed the full
average, $F\geq0$ implies $\kappa\geq0$. If $\kappa=0$, the zero-trace
observation following \eqref{eq:kyfan} implies $\mathring{W}+\beta_{n}\Id=0$,
because $k<\dim S^{2}_{0}(T_{x}M)$. Thus $\mathring{W}=0$, and \eqref{eq:weyl-traces}
gives $W=0$. The metric is flat, so completeness and simple connectivity
identify $M$ with Euclidean space.

Assume henceforth that $\kappa>0$. Bonnet-Myers makes $M$ compact.
A complete simply connected locally symmetric manifold is globally
symmetric, and its de Rham factors are compact irreducible symmetric
spaces with the same Einstein constant $\rho=(n-1)\kappa$.

\smallskip{}
\noindent \emph{Reducible spaces.} Write $M=M^{p}_{1}\times M^{q}_{2}$,
allowing either factor initially to be reducible. Positive Ricci curvature
implies $p,q\geq2$. Set $c=\frac{2\rho}{n+2}$. Both factors are
Einstein metrics with Einstein constant $\rho$. Thus the trace free
block diagonal symmetric 2-tensor $S_{12}=\frac{1}{p}g_{1}-\frac{1}{q}g_{2}$
has $\mathring{R}$-eigenvalue $-\rho$, and the $pq$-dimensional
mixed subspace has eigenvalue zero. Since $pq\geq k-1$, the variational
principle \eqref{eq:kyfan} gives 
\[
0\leq F\leq(c-\rho)+(k-1)c=0.
\]
If one factor had a negative second-kind eigenvalue $\lambda$, extend
its trace-free eigenvector to be zero to the other factor. This is
an eigenvector of the product, orthogonal to both subspaces just used.
Assigning weight one to it and to the block-metric line, and total
weight $k-2$ to mixed tensors, would give 
\[
F\leq(c-\rho)+(c+\lambda)+(k-2)c=\lambda<0.
\]
Thus both factors have nonnegative curvature operator of the second
kind. In dimension at least four, the Einstein rigidity theorem of
Dai-Fu \cite{DaiFu} makes each factor a space of constant curvature;
in dimensions two and three this follows directly from the Einstein
condition. The factors are complete and simply connected, so they
are round spheres. Conversely, an Einstein product of two round spheres
has precisely the eigenvalue $-\rho$, eigenvalue 0 on the mixed subspace,
and positive eigenvalues on the two trace-free factor subspaces. Hence
its actual lowest $k$ sum is $(c-\rho)+(k-1)c=0$.

\smallskip{}
\noindent \emph{The exclusion test for irreducible spaces.} Normalize
$\rho=\frac{1}{2}$ and write $d=\dim M$. Then 
\[
F=\Sigma_{\frac{d+2}{2}}(\mathring{R})+\frac{1}{2}.
\]
An eigenvalue $\lambda<0$ occurring on a subspace of dimension $m$
therefore excludes a space as soon as 
\begin{equation}
m\geq\frac{d+2}{2},\quad(d+2)|\lambda|>1.\label{eq:symmetric-exclusion}
\end{equation}
Indeed, choose a symmetric endomorphism $P$ with $0\leq P\leq\Id$
and $\tr P=\frac{d+2}{2}$, supported on that eigenspace. The variational
formula then gives $F\leq\frac{(d+2)\lambda}{2}+\frac{1}{2}<0$. The
eigenvalue need not be the lowest eigenvalue.

For $SU(r)$ ($r\geq3$), $SU(r)/SO(r)$ ($r\geq3$), and $SU(2r)/Sp(r)$
($r\geq3$), the eigenvalue $-\frac{1}{4}$ has multiplicity $d\geq5$,
so \eqref{eq:symmetric-exclusion} applies. For the other classical
families, use the following entries of Appendix~\ref{app:symmetric-spectra}:
the $S^{2}_{0}\mathbb{R}^{r}$ summand for $\operatorname{Spin}(r)$;
the $\omega_{2}$ summand for $Sp(r)$; the $\bigwedge^{2}\mathbb{R}^{p}\otimes\bigwedge^{2}\mathbb{R}^{q}$
summand for real Grassmannians; the summand associated with the larger
factor for unequal complex and quaternionic Grassmannians, and both
such summands when the factors have equal dimension; and the adjoint
$SU(r)$ summand for $Sp(r)/U(r)$ and $SO(2r)/U(r)$.

For completeness, the inequalities in \eqref{eq:symmetric-exclusion}
can be checked without estimating the remaining spectrum. For $\operatorname{Spin}(r)$,
$r\geq6$, one has $m\geq d\geq15$ and $|\lambda|\geq\frac{1}{8}$.
For $Sp(r)$, $r\geq3$, 
\[
m-\frac{d+2}{2}=\frac{2r^{2}-3r-4}{2}>0,\quad|\lambda|>\frac{1}{4}.
\]
For real Grassmannians with $p,q\geq3$, $(p-1)(q-1)\geq4$ gives
$m\geq d$, and $d+2>p+q-2$; when $p=2$, $q\geq4$ gives $\frac{q(q-1)}{2}\geq q+1$
and $\frac{2q+2}{q}>1$. For complex Grassmannians with $2\leq p<q$
and quaternionic Grassmannians with $1\leq p<q$, respectively, 
\begin{align*}
m-\frac{d+2}{2} & =q(q-p)-2>0, & (d+2)|\lambda|-1 & =\frac{q(p^{2}-1)}{p+q}>0,\\
m-\frac{d+2}{2} & =2q(q-p)-q-2\geq0, & (d+2)|\lambda|-1 & =\frac{q(2p(p+1)-1)}{p+q+1}>0.
\end{align*}
For equal parameters $p=q=r\geq2$, the multiplicity differences are
$r^{2}-3$ and $2r^{2}-2r-3$, and $|\lambda|\geq\frac{1}{4}$. For
$Sp(r)/U(r)$, $r\geq3$, and $SO(2r)/U(r)$, $r\geq4$, they are
$\frac{r^{2}-r-4}{2}$ and $\frac{r^{2}+r-4}{2}$; their dimensions
are at least twelve and $|\lambda|\geq\frac{1}{6}$.

Two small nonspherical cases require using two negative eigenspaces.
For $Sp(2)\simeq\operatorname{Spin}(5)$, $d=10$, the eigenvalues
$-\frac{1}{3}$ and $-\frac{1}{12}$ occur with multiplicities five
and fourteen. The modules are $\bigwedge^{4}\mathbb{R}^{5}$ and $S^{2}_{0}\mathbb{R}^{5}$.
Consequently 
\[
F\leq5\left(-\frac{1}{3}\right)-\frac{1}{12}+\frac{1}{2}=-\frac{5}{4}<0.
\]
For $SO(5)/(SO(2)SO(3))$, locally isometric to $Sp(2)/U(2)$, $d=6$,
and eigenvalues $-\frac{1}{3}$ and $-\frac{1}{6}$ occur on $\bigwedge^{2}\mathbb{R}^{2}\otimes\bigwedge^{2}\mathbb{R}^{3}$
and $S^{2}_{0}\mathbb{R}^{2}$, of dimensions three and two. Thus
\[
F\leq3\left(-\frac{1}{3}\right)-\frac{1}{6}+\frac{1}{2}=-\frac{2}{3}<0.
\]

For every exceptional compact group, the negative eigenvalue and its
multiplicity in Table~\ref{tab:appendix-groups} satisfy \eqref{eq:symmetric-exclusion}.
The same is true of the lowest eigenvalue in each row of Table~\ref{tab:appendix-exceptional}.
In all twelve quotient cases, $m\geq\frac{d+2}{2}$ and $(d+2)|\lambda|\geq\frac{5}{2}$;
equality in the latter bound occurs for $G_{2}/SO(4)$. Thus all exceptional
cases are excluded, including the Cayley plane.

By the classification of compact irreducible symmetric spaces, the
remaining cases are round spheres and complex projective spaces. The
low-parameter identifications include $SU(4)/Sp(2)=S^{5}$ and $SO(6)/U(3)=\mathbb{CP}^{3}$;
the real Grassmannian with $(p,q)=(2,2)$ is reducible and was already
treated. On a round sphere, 
\[
F=\frac{d+2}{2}\left(\frac{\rho}{d-1}+\frac{2\rho}{d+2}\right)=\frac{3d\rho}{2(d-1)}>0.
\]
On $\mathbb{CP}^{a}$, $d=2a$, with $\rho=\frac{1}{2}$, the lowest
second-kind eigenvalue is $-\frac{1}{2(a+1)}=-\frac{1}{d+2}$, of
multiplicity $a^{2}-1\geq a+1$ for $a\geq2$; the other eigenvalue
is positive. Thus the lowest eigenvalue of the shifted operator is
zero with multiplicity at least $k$, and $F=0$. These calculations
establish both directions and the stated boundary values. 
\end{proof}

\appendix
%dummy comment inserted by tex2lyx to ensure that this paragraph is not empty

\section{COSK eigenvalues of compact symmetric spaces}

\label{app:symmetric-spectra}

We calculate the complete spectrum of $\mathring{R}$ on every compact
irreducible symmetric space, with the uniform normalization 
\begin{equation}
\Ric=\frac{1}{2}g.\label{eq:appendix-normalization}
\end{equation}
The COSK calculations go back to Koiso \cite[Section~7]{Koiso1980}.
We include the isotropy decompositions and the real multiplicities,
both to fix the normalization and to make the eigenvalue tests in
Section~\ref{sec:symmetric-classification} explicit. We also correct
a factor-of-two error in the type-$A$ group entry of Koiso's table;
see Remark~\ref{rem:Koiso-correction}. These tables complement the
spectra of the ordinary curvature operator on $\bigwedge^{2}$ computed
by Petersen-Stanfield \cite{PetersenStanfield2025} for all compact
irreducible symmetric spaces, with the same normalization $\Ric=\frac{1}{2}g$.
The last subsection gives direct calculations for Einstein products
and complex projective spaces.

The tables also apply to finite coverings and locally isometric quotients
as they clearly have the same spectrum. Finite central quotients in
the notation $G/K$ are therefore suppressed. Every multiplicity is
a real dimension, and the eigenvalue $-\frac{1}{2}$ that comes from
the metric is excluded. An entry $\lambda^{[m]}$ means eigenvalue
$\lambda$ with multiplicity $m$. There is no one-dimensional example
satisfying \eqref{eq:appendix-normalization}.

\subsection{The curvature-Casimir identity}

\label{sec:appendix-casimir}

Write the effective transvection algebra as $\mathfrak{g}=\mathfrak{k}\oplus\mathfrak{p}$
and equip $\mathfrak{p}$ with $-B_{\mathfrak{g}}$. This is precisely
the metric \eqref{eq:appendix-normalization}. If $u_{a}$ is an orthonormal
basis of $\mathfrak{k}$ for the restricted form $-B_{\mathfrak{g}}$,
put $A_{a}=\ad(u_{a})|_{\mathfrak{p}}$. These endomorphisms are skew-adjoint.
The positive isotropy Casimir operator on symmetric endomorphisms
is 
\[
C_{K}(H)=-\sum_{a}[A_{a},[A_{a},H]].
\]
The symmetric-space curvature formula gives $\widehat{R}(H)=\sum_{a}A_{a}HA_{a}$,
while its Ricci contraction gives $-\sum_{a}A^{2}_{a}=\frac{1}{2}\Id$.
Expanding the commutators therefore yields 
\[
C_{K}(H)=-\sum_{a}(A^{2}_{a}H+HA^{2}_{a}-2A_{a}HA_{a})=H+2\widehat{R}(H).
\]
Consequently, 
\begin{equation}
\mathring{R}=\frac{1}{2}(C_{K}-\Id)\quad\text{on }S^{2}_{0}\mathfrak{p},\quad\lambda_{i}=\frac{c_{i}-1}{2},\label{eq:appendix-casimir}
\end{equation}
where $c_{i}$ correspond to the eigenvalues of $C_{K}$ and $\lambda_{i}$
are the eigenvalues of $\mathring{R}$. Thus the calculation consists
of decomposing $S^{2}_{0}\mathfrak{p}$ into irreducible components
and noting that on each summand the Casimir operator is given by $c\Id$.
The summand coming from the metric has Casimir eigenvalues $c=0$ and curvature eigenvalue $-\frac{1}{2}$ and is removed from the considerations.

We recall how to recover the numerical data. For a simple ideal $\mathfrak{k}_{i}$
we have that $t_{i}B_{\mathfrak{g}}|_{\mathfrak{k}_{i}}=B_{\mathfrak{k}_{i}}$.
If $\mu_{i}$ is a highest weight, $\delta_{i}$ is half the sum of
its positive roots, and $\theta_{i}$ is its highest root, then 
\begin{equation}
c_{i}=t_{i}\frac{(\mu_{i},\mu_{i}+2\delta_{i})}{(\theta_{i},\theta_{i}+2\delta_{i})},\quad\dim V_{\mu_{i}}=\prod_{\alpha>0}\frac{(\mu_{i}+\delta_{i},\alpha)}{(\delta_{i},\alpha)}.\label{eq:appendix-highest-weights}
\end{equation}
Casimir operators add on tensor products of different isotropy factors;
a central factor contributes the squared norm of its weight for the
restricted metric. Dimensions multiply on tensor products and add
on direct sums. The complex dimension of a complexified real eigenspace
equals its original real dimension. In particular, if $\mathfrak{p}_{\mathbb{C}}=U\oplus U^{*}$,
then 
\[
S^{2}_{0}\mathfrak{p}_{\mathbb{C}}=S^{2}U\oplus S^{2}U^{*}\oplus(U\otimes U^{*})_{0}.
\]
A summand of $S^{2}U$ and its conjugate contribute twice its complex
dimension to the real multiplicity. Fundamental weights below use
Bourbaki numbering.

Two useful checks, with $d=\dim M$, are 
\begin{equation}
\sum m_{\lambda_{i}}=\frac{(d-1)(d+2)}{2},\quad\sum m_{\lambda_{i}}\lambda_{i}=\frac{d+2}{4}.\label{eq:appendix-checks}
\end{equation}
For the second identity, the trace of $H\mapsto AHA$ on $S^{2}$
is $-\frac{1}{2}\tr A^{2}$ for skew-adjoint $A$. Thus $\tr_{S^{2}}\widehat{R}=\frac{d}{4}$;
deleting the metric eigenvalue $-\frac{1}{2}$ gives the asserted
trace. As a normalization check, the round $d$-sphere has the single
eigenvalue $\frac{1}{2(d-1)}$, of multiplicity $\frac{(d-1)(d+2)}{2}$.
If a reference metric has Einstein constant $\rho_{0}$, then the
eigenvalues must be divided by $2\rho_{0}$ to obtain \eqref{eq:appendix-normalization}.

\subsection{Compact simple Lie groups}

\label{sec:appendix-groups}

For a compact simple group $H$, the symmetric presentation $(H\times H)/\operatorname{diag}H$
gives the bi-invariant metric $g=-\frac{1}{2}B_{\mathfrak{h}}$. Let
$C_{H}$ denote the Casimir operator for $-B_{\mathfrak{h}}$, so that its
eigenvalue on the adjoint representation is one. In this notation
\begin{equation}
\mathring{R}=\frac{1}{4}(C_{H}-2\Id).\label{eq:appendix-group-casimir}
\end{equation}
Indeed, for a $-B_{\mathfrak{h}}$-orthonormal basis with adjoint
endomorphisms $A_{a}$, one has $C_{H}(T)=2T+2\sum_{a}A_{a}TA_{a}$,
whereas the curvature of the metric $-\frac{1}{2}B_{\mathfrak{h}}$
gives $\widehat{R}(T)=\frac{1}{2}\sum_{a}A_{a}TA_{a}$.

The symmetric-square decompositions of the adjoint representation
(cf.\ \cite{LandsbergManivel2006}) give the following complete list.
\begingroup 
\small
\setlength{\tabcolsep}{5pt}
\renewcommand{\arraystretch}{1.65}
%\begingroup {\small\setlength{\tabcolsep}{5pt} 
%\global\long\def\arraystretch{1.65}%
% }{\small{}%
\begin{longtable}[c]{|l|l|c|c|} 
\caption{Compact simple groups, normalized by $\Ric=\frac{1}{2}g$.}\label{tab:appendix-groups}
\tabularnewline
\hline 
{\small Group } & {\small$d$ } & {\small Eigenvalue } & {\small Multiplicity}\tabularnewline
\endfirsthead
\hline 
\multicolumn{4}{c}{{\small Table~\thetable\ (continued)}}\tabularnewline
\hline 
{\small Group } & {\small$d$ } & {\small Eigenvalue } & {\small Multiplicity}\tabularnewline
\endhead
\hline 
{\small$SU(2)$ } & {\small$3$ } & {\small$\frac{1}{4}$ } & {\small$5$}\tabularnewline
\hline 
{\small$SU(r)$, $r\geq3$ } & {\small$r^{2}-1$ } & {\small$-\frac{1}{4}$ } & {\small$r^{2}-1$}\tabularnewline
{\small{*}} &  & {\small$-\frac{1}{2r}$ } & {\small$\frac{r^{2}(r+1)(r-3)}{4}$}\tabularnewline
{\small{*}} &  & {\small$\frac{1}{2r}$ } & {\small$\frac{r^{2}(r-1)(r+3)}{4}$}\tabularnewline
\hline 
{\small$\operatorname{Spin}(r)$ } & {\small$\frac{r(r-1)}{2}$ } & {\small$-\frac{1}{r-2}$ } & {\small$\binom{r}{4}$}\tabularnewline
{\small{*}$r\geq5$ } &  & {\small$-\frac{r-4}{4(r-2)}$ } & {\small$\frac{(r-1)(r+2)}{2}$}\tabularnewline
{\small{*}} &  & {\small$\frac{1}{2(r-2)}$ } & {\small$\frac{r(r+1)(r+2)(r-3)}{12}$}\tabularnewline
\hline 
{\small$Sp(r)$, $r\geq2$ } & {\small$r(2r+1)$ } & {\small$-\frac{r+2}{4(r+1)}$ } & {\small$(r-1)(2r+1)$}\tabularnewline
{\small{*}} &  & {\small$-\frac{1}{4(r+1)}$ } & {\small$\frac{r(r-1)(2r-1)(2r+3)}{3}$}\tabularnewline
{\small{*}} &  & {\small$\frac{1}{2(r+1)}$ } & {\small$\binom{2r+3}{4}$}\tabularnewline
\hline 
{\small$G_{2}$ } & {\small$14$ } & {\small$-\frac{5}{24},\ \frac{1}{8}$ } & {\small$27,\ 77$}\tabularnewline
\hline 
{\small$F_{4}$ } & {\small$52$ } & {\small$-\frac{5}{36},\ \frac{1}{18}$ } & {\small$324,\ 1053$}\tabularnewline
\hline 
{\small$E_{6}$ } & {\small$78$ } & {\small$-\frac{1}{8},\ \frac{1}{24}$ } & {\small$650,\ 2430$}\tabularnewline
\hline 
{\small$E_{7}$ } & {\small$133$ } & {\small$-\frac{1}{9},\ \frac{1}{36}$ } & {\small$1539,\ 7371$}\tabularnewline
\hline 
{\small$E_{8}$ } & {\small$248$ } & {\small$-\frac{1}{10},\ \frac{1}{60}$ } & {\small$3875,\ 27000$}\tabularnewline
\hline 
\end{longtable}
\par
\endgroup
%\end{longtable}}{\small\endgroup}{\small\par}

{\small Here is a verification of the entries. For $SU(r)$ the three
summands, in table order, have highest weights $\theta$, $\omega_{2}+\omega_{r-2}$,
and $2\theta$. Their $C_{H}$ eigenvalues are $1$, $2-\frac{2}{r}$,
and $2+\frac{2}{r}$. The middle summand is absent for $r=3$ and
has weight $2\omega_{2}$ for $r=4$. The eigenvalue $-\frac{1}{4}$
on the adjoint summand follows directly from \eqref{eq:appendix-group-casimir};
in particular, this value does not depend on a choice of root normalization.}{\small\par}

{\small For $\operatorname{Spin}(r)$ the summands are $\bigwedge^{4}\mathbb{R}^{r}$,
$S^{2}_{0}\mathbb{R}^{r}$, and the space of algebraic Weyl tensors
on $\mathbb{R}^{r}$. Their Casimir eigenvalues are $2-\frac{4}{r-2}$,
$\frac{r}{r-2}$, and $2+\frac{2}{r-2}$. For $Sp(r)$ the highest
weights are $\omega_{2}$, $2\omega_{2}$, and $4\omega_{1}$, with
Casimir eigenvalues $\frac{r}{r+1}$, $2-\frac{1}{r+1}$, and $2+\frac{2}{r+1}$.
Formula~\eqref{eq:appendix-group-casimir} now gives all the classical
eigenvalues; the Weyl dimension formula gives the displayed multiplicities.}{\small\par}

{\small For each exceptional group the trace-free symmetric square
has two summands. The positive summand is the Cartan square of the
adjoint, with Casimir eigenvalues $2+\frac{2}{h^{\vee}}$, where $h^{\vee}=4,9,12,18,30$
for $G_{2},F_{4},E_{6},E_{7},E_{8}$, respectively. The other summand
has Casimir eigenvalues $\frac{7}{6},\frac{13}{9},\frac{3}{2},\frac{14}{9},\frac{8}{5}$,
respectively. Applying \eqref{eq:appendix-group-casimir} gives the
two eigenvalues. Their multiplicities can also be recovered from \eqref{eq:appendix-checks}:
if the eigenvalues are $\lambda_{-}<0<\lambda_{+}$, then 
\[
m_{-}=\frac{\frac{1}{2}(d-1)(d+2)\lambda_{+}-\frac{d+2}{4}}{\lambda_{+}-\lambda_{-}},\quad m_{+}=\frac{(d-1)(d+2)}{2}-m_{-}.
\]
}{\small\par}

{\small The low-rank identifications $Sp(1)\simeq SU(2)\simeq\operatorname{Spin}(3)$,
$Sp(2)\simeq\operatorname{Spin}(5)$, and $SU(4)\simeq\operatorname{Spin}(6)$
agree with the table. The group $\operatorname{Spin}(4)$ is reducible.
At $r=8$, the two negative orthogonal-group entries merge: the spectrum
is $(-\frac{1}{6})^{[105]}$, $(\frac{1}{12})^{[300]}$. Although
$\bigwedge^{4}\mathbb{R}^{8}$ splits into two summands, both have
the same Casimir eigenvalue, so no additional eigenvalue appears.}{\small\par}

{\small\begin{remark}[Correction of Koiso's type-$A$ group entry]
\label{rem:Koiso-correction} In the group-type $A_{r}$ row of \cite[p.~69]{Koiso1980},
corresponding to $SU(r+1)$ with $r\geq2$, the curvature eigenvalue
on the adjoint summand is printed as $-(r+1)$, and the Einstein constant
is $\varepsilon=r+1$. The eigenvalue should be $-\frac{r+1}{2}$.
Indeed, the adjoint Casimir operator satisfies $C_{H}=\Id$, so \eqref{eq:appendix-group-casimir}
gives the normalized eigenvalue $-\frac{1}{4}$. Multiplication by
$2\varepsilon$ recovers the corrected entry in Koiso's scale.\par The
trace identity gives an independent check. For $SU(3)$, $S^{2}_{0}\mathfrak{su}(3)$
decomposes into modules of dimensions $8$ and $27$, and the positive
eigenvalue is $\frac{1}{6}$. Rescaling the printed entry would give
$-\frac{1}{2}$ on the first module, whereas 
\[
8\left(-\frac{1}{2}\right)+27\left(\frac{1}{6}\right)=\frac{1}{2}\ne\frac{8+2}{4}.
\]
The corrected eigenvalue gives $8(-\frac{1}{4})+27(\frac{1}{6})=\frac{5}{2}$,
exactly the trace required by \eqref{eq:appendix-checks}. Table~\ref{tab:appendix-groups}
uses this corrected value. \end{remark}}{\small\par}

\subsection{Classical irreducible quotients}

{\small\label{sec:appendix-classical}}{\small\par}

{\small All metrics in this subsection satisfy $\Ric=\frac{1}{2}g$.
The multiplicities are real dimensions. In each display, omit a row
of multiplicity zero and add the multiplicities of coincident eigenvalues.
We write $d=\dim M$ and suppress finite ineffective subgroups of
the isotropy group. The modules below are complexified; their complex
dimensions give the real multiplicities, with both members of a conjugate
pair included. Each spectrum follows from \eqref{eq:appendix-casimir},
namely $\lambda=\frac{c-1}{2}$ for an isotropy-Casimir eigenvalue
$c$.}{\small\par}

{\small\paragraph{BDI: real Grassmannians.} Let 
\[
M=SO(p+q)/(SO(p)\times SO(q)),\quad d=pq,\quad D=p+q-2,
\]
where $p,q\geq1$, $p+q\geq3$, and $(p,q)\ne(2,2)$. The excluded
pair is locally a product of two round spheres. Set $s_{r}=\frac{(r-1)(r+2)}{2}$
and $a_{r}=\frac{r(r-1)}{2}$. Then 
\[
{
\global\long\def\arraystretch{1.5}%
\begin{array}{|c|c|}
\hline \text{Eigenvalue} & \text{Multiplicity}\\[3pt]
\hline \frac{2-q}{2D} & s_{p}\\[3pt]
\hline \frac{2-p}{2D} & s_{q}\\[3pt]
\hline -\frac{1}{D} & a_{p}a_{q}\\[3pt]
\hline \frac{1}{D} & s_{p}s_{q}\\[3pt]
\hline \end{array}}
\]
To see this, put $E=\mathbb{C}^{p}$, $F=\mathbb{C}^{q}$, and let
$\mathcal{H}^{j}E$ denote the trace-free symmetric tensors of degree
$j$. Since $\mathfrak{p}_{\mathbb{C}}=E\otimes F$, the identity
\[
S^{2}(E\otimes F)=(S^{2}E\otimes S^{2}F)\oplus(\sideset{}{^{2}}\bigwedge E\otimes\sideset{}{^{2}}\bigwedge F)
\]
gives, after deleting the metric, 
\[
S^{2}_{0}\mathfrak{p}_{\mathbb{C}}=\mathcal{H}^{2}E\oplus\mathcal{H}^{2}F\oplus(\sideset{}{^{2}}\bigwedge E\otimes\sideset{}{^{2}}\bigwedge F)\oplus(\mathcal{H}^{2}E\otimes\mathcal{H}^{2}F).
\]
In the metric induced from $-B_{\mathfrak{so}(p+q)}$, the Casimir
operator on an $SO(r)$ tensor module of highest weight $\mu$ is
\[
c_{r}(\mu)=\frac{\sum_{i}\mu_{i}(\mu_{i}+r-2i)}{2D}.
\]
Thus $c(\mathcal{H}^{2}\mathbb{C}^{r})=\frac{r}{D}$ and $c(\bigwedge^{2}\mathbb{C}^{r})=\frac{r-2}{D}$;
the Casimir operators of the two factors add. The dimensions are $s_{r}$
and $a_{r}$, respectively. Any splitting of these modules in small
dimensions leaves the displayed eigenvalues and total multiplicities
unchanged. If one parameter is one, the table reduces to the round-sphere
eigenvalue $\frac{1}{2(d-1)}$. For $p,q\geq2$, coincidences occur
precisely when $p=q$, $p=4$, or $q=4$.}{\small\par}

{\small\paragraph{AIII: complex Grassmannians.} For $p,q\geq1$,
put 
\[
M=SU(p+q)/S(U(p)\times U(q)),\quad d=2pq,\quad P=p+q.
\]
The spectrum is 
\[
{
\global\long\def\arraystretch{1.5}%
\begin{array}{|c|c|}
\hline \text{Eigenvalue} & \text{Multiplicity}\\[3pt]
\hline -\frac{q}{2P} & p^{2}-1\\[3pt]
\hline -\frac{p}{2P} & q^{2}-1\\[3pt]
\hline 0 & (p^{2}-1)(q^{2}-1)\\[3pt]
\hline -\frac{1}{P} & \frac{pq(p-1)(q-1)}{2}\\[3pt]
\hline \frac{1}{P} & \frac{pq(p+1)(q+1)}{2}\\[3pt]
\hline \end{array}}
\]
Write $U=E\otimes F^{*}$, where $E=\mathbb{C}^{p}$ and $F=\mathbb{C}^{q}$.
The Hermitian isotropy decomposition is 
\begin{equation}
S^{2}_{0}\mathfrak{p}_{\mathbb{C}}=S^{2}U\oplus S^{2}U^{*}\oplus(U\otimes U^{*})_{0}.\label{eq:appendix-hermitian-decomposition}
\end{equation}
Here 
\begin{align*}
S^{2}U & =(S^{2}E\otimes S^{2}F^{*})\oplus(\bigwedge^{2}E\otimes\bigwedge^{2}F^{*}),\\
(U\otimes U^{*})_{0} & =\mathfrak{sl}(E)\oplus\mathfrak{sl}(F)\oplus(\mathfrak{sl}(E)\otimes\mathfrak{sl}(F)).
\end{align*}
For a highest weight $(u;v)$ with $\sum_{i}u_{i}+\sum_{j}v_{j}=0$,
the induced Casimir operator is 
\[
c(u;v)=\frac{\sum_{i}u_{i}(u_{i}+p+1-2i)+\sum_{j}v_{j}(v_{j}+q+1-2j)}{2P}.
\]
The three Hermitian summands have Casimir eigenvalues $\frac{p}{P}$,
$\frac{q}{P}$, and $1$. The two summands of $S^{2}U$ have Casimir
eigenvalues $1+\frac{2}{P}$ and $1-\frac{2}{P}$, respectively; their
conjugates double their dimensions. This gives all five rows. If $p$
or $q$ are 1, this is the complex-projective-space spectrum. For
$p,q\geq2$, negative eigenvalues coincide precisely when $p=q$,
$p=2$, or $q=2$; there are five distinct eigenvalues exactly when
$p,q\geq3$ and $p\ne q$.}{\small\par}

{\small\paragraph{CII: quaternionic Grassmannians.} Let 
\[
M=Sp(p+q)/(Sp(p)\times Sp(q)),\quad d=4pq,\quad L=p+q+1,
\]
where $p,q\geq1$. Set $b_{r}=(r-1)(2r+1)$ and $z_{r}=r(2r+1)$.
Then 
\[
{
\global\long\def\arraystretch{1.5}%
\begin{array}{|c|c|}
\hline \text{Eigenvalue} & \text{Multiplicity}\\[3pt]
\hline -\frac{q+1}{2L} & b_{p}\\[3pt]
\hline -\frac{p+1}{2L} & b_{q}\\[3pt]
\hline -\frac{1}{2L} & b_{p}b_{q}\\[3pt]
\hline \frac{1}{2L} & z_{p}z_{q}\\[3pt]
\hline \end{array}}
\]
Take $E=\mathbb{C}^{2p}$ and $F=\mathbb{C}^{2q}$ with their symplectic
forms. Since $\mathfrak{p}_{\mathbb{C}}=E\otimes F$ and $\bigwedge^{2}E=\mathbb{C}\oplus\bigwedge^{2}_{0}E$,
\[
S^{2}_{0}\mathfrak{p}_{\mathbb{C}}=\sideset{}{_{0}^{2}}\bigwedge E\oplus\sideset{}{_{0}^{2}}\bigwedge F\oplus(\sideset{}{_{0}^{2}}\bigwedge E\otimes\sideset{}{_{0}^{2}}\bigwedge F)\oplus(S^{2}E\otimes S^{2}F).
\]
The induced $Sp(r)$ Casimir operator is 
\[
c_{r}(\mu)=\frac{\sum_{i}\mu_{i}(\mu_{i}+2r+2-2i)}{4L}.
\]
Its values on $\bigwedge^{2}_{0}\mathbb{C}^{2r}$ and $S^{2}\mathbb{C}^{2r}$
are $\frac{r}{L}$ and $\frac{r+1}{L}$, while the dimensions are
$b_{r}$ and $z_{r}$. This proves the table. The first two rows merge
when $p=q$. A parameter equal to one gives a quaternionic projective
space; $(p,q)=(1,1)$ gives the round four-sphere.}{\small\par}

{\small\paragraph{AI: $SU(r)/SO(r)$.} For $r\geq3$, set $d=\frac{(r-1)(r+2)}{2}$.
The complete spectrum is 
\[
{
\global\long\def\arraystretch{1.5}%
\begin{array}{|c|c|}
\hline \text{Eigenvalue} & \text{Multiplicity}\\[3pt]
\hline -\frac{1}{4} & \frac{(r-1)(r+2)}{2}\\[3pt]
\hline -\frac{1}{2r} & \frac{r(r+1)(r+2)(r-3)}{12}\\[3pt]
\hline \frac{1}{r} & \frac{r(r+1)(r-1)(r+6)}{24}\\[3pt]
\hline \end{array}}
\]
For the standard orthogonal module $E=\mathbb{C}^{r}$, $\mathfrak{p}_{\mathbb{C}}=\mathcal{H}^{2}E$
and 
\[
S^{2}_{0}(\mathcal{H}^{2}E)=\mathcal{H}^{2}E\oplus\mathcal{W}(E)\oplus\mathcal{H}^{4}E.
\]
Here $\mathcal{W}(E)$ is the algebraic Weyl-tensor module, of highest
weight $(2,2)$, and $\mathcal{H}^{4}E$ has dimension $\binom{r+3}{4}-\binom{r+1}{2}$.
The Casimir operator 
\[
c(\mu)=\frac{\sum_{i}\mu_{i}(\mu_{i}+r-2i)}{4r}
\]
gives $\frac{1}{2}$, $1-\frac{1}{r}$, and $1+\frac{2}{r}$ on these
three summands. At $r=3$ the Weyl-tensor summand is absent; at $r=4$
it splits into two summands with the same Casimir eigenvalue. The case
$r=2$ is the round two-sphere, with eigenvalue $\frac{1}{2}$ of
multiplicity two.}{\small\par}

{\small\paragraph{AII: $SU(2r)/Sp(r)$.} For $r\geq3$, set $d=(r-1)(2r+1)$.
Then 
\[
{
\global\long\def\arraystretch{1.5}%
\begin{array}{|c|c|}
\hline \text{Eigenvalue} & \text{Multiplicity}\\[3pt]
\hline -\frac{1}{4} & (r-1)(2r+1)\\[3pt]
\hline -\frac{1}{2r} & \binom{2r}{4}-\binom{2r}{2}\\[3pt]
\hline \frac{1}{4r} & \frac{r(r-1)(2r-1)(2r+3)}{3}\\[3pt]
\hline \end{array}}
\]
For the standard symplectic module $E=\mathbb{C}^{2r}$, $\mathfrak{p}_{\mathbb{C}}=\bigwedge^{2}_{0}E$,
and symplectic contraction gives 
\[
S^{2}_{0}(\sideset{}{_{0}^{2}}\bigwedge E)=\sideset{}{_{0}^{2}}\bigwedge E\oplus\sideset{}{_{0}^{4}}\bigwedge E\oplus V_{(2,2)}.
\]
The middle summand is present only for $r\geq4$; $V_{(2,2)}$ has
highest weight $2e_{1}+2e_{2}$ and the dimension in the last row.
The Casimir formula 
\[
c(\mu)=\frac{\sum_{i}\mu_{i}(\mu_{i}+2r+2-2i)}{8r}
\]
gives $\frac{1}{2}$, $1-\frac{1}{r}$, and $1+\frac{1}{2r}$, respectively.
For $r=2$, only the last summand remains: the space is the round
five-sphere, with eigenvalue $\frac{1}{8}$ of multiplicity fourteen.
The case $r=1$ is a point and is excluded.}{\small\par}

{\small\paragraph{CI: $Sp(r)/U(r)$.} For $r\geq1$, the dimension
is $d=r(r+1)$, and 
\[
{
\global\long\def\arraystretch{1.5}%
\begin{array}{|c|c|}
\hline \text{Eigenvalue} & \text{Multiplicity}\\[3pt]
\hline -\frac{r+2}{4(r+1)} & r^{2}-1\\[3pt]
\hline 0 & \frac{r^{2}(r-1)(r+3)}{4}\\[3pt]
\hline -\frac{1}{2(r+1)} & \frac{r^{2}(r^{2}-1)}{6}\\[3pt]
\hline \frac{1}{r+1} & 2\binom{r+3}{4}\\[3pt]
\hline \end{array}}
\]
In \eqref{eq:appendix-hermitian-decomposition}, take $U=S^{2}E$,
$E=\mathbb{C}^{r}$. For $r\geq2$, 
\begin{align*}
S^{2}U & =S^{4}E\oplus S_{(2,2)}E,\\
(U\otimes U^{*})_{0} & =\mathfrak{sl}(E)\oplus V_{2e_{1}-2e_{r}}.
\end{align*}
The symbol $S_{(2,2)}E$ denotes the Schur module of partition $(2,2)$,
of dimension $\frac{r^{2}(r^{2}-1)}{12}$. The last Hermitian summand
has dimension $[\frac{r(r+1)}{2}]^{2}-r^{2}$. For the induced $U(r)$
metric, 
\[
c(\mu)=\frac{\sum_{i}\mu_{i}(\mu_{i}+r+1-2i)}{4(r+1)}.
\]
The weights $4e_{1}$, $2e_{1}+2e_{2}$, $e_{1}-e_{r}$, and $2e_{1}-2e_{r}$
have Casimir eigenvalues $\frac{r+3}{r+1}$, $\frac{r}{r+1}$, $\frac{r}{2(r+1)}$,
and $1$. Doubling the dimensions of the first two summands accounts
for their conjugates. At $r=1$ only the last row survives, giving
the round two-sphere; for $r\geq2$ all four eigenvalues are distinct.}{\small\par}

{\small\paragraph{DIII: $SO(2r)/U(r)$.} For $r\geq3$, the dimension
is $d=r(r-1)$, and 
\[
{
\global\long\def\arraystretch{1.5}%
\begin{array}{|c|c|}
\hline \text{Eigenvalue} & \text{Multiplicity}\\[3pt]
\hline -\frac{r-2}{4(r-1)} & r^{2}-1\\[3pt]
\hline 0 & \frac{r^{2}(r+1)(r-3)}{4}\\[3pt]
\hline -\frac{1}{r-1} & 2\binom{r}{4}\\[3pt]
\hline \frac{1}{2(r-1)} & \frac{r^{2}(r^{2}-1)}{6}\\[3pt]
\hline \end{array}}
\]
Now $U=\bigwedge^{2}E$ in \eqref{eq:appendix-hermitian-decomposition}.
For $r\geq4$, 
\begin{align*}
S^{2}U & =S_{(2,2)}E\oplus\sideset{}{^{4}}\bigwedge E,\\
(U\otimes U^{*})_{0} & =\mathfrak{sl}(E)\oplus V_{e_{1}+e_{2}-e_{r-1}-e_{r}}.
\end{align*}
The last summand has dimension $[\frac{r(r-1)}{2}]^{2}-r^{2}$. The
induced Casimir operator is 
\[
c(\mu)=\frac{\sum_{i}\mu_{i}(\mu_{i}+r+1-2i)}{4(r-1)}.
\]
On the two summands of $S^{2}U$ its values are $\frac{r}{r-1}$ and
$\frac{r-3}{r-1}$; on the two Hermitian summands they are $\frac{r}{2(r-1)}$
and $1$. Include the conjugates of $S^{2}U$ to obtain the displayed
multiplicities. At $r=3$ only the first and last rows remain, agreeing
with $\mathbb{CP}^{3}$. The negative rows merge at $r=6$. The case
$r=2$ is the round two-sphere, with eigenvalue $\frac{1}{2}$ of
multiplicity two.}{\small\par}

{\small In every permitted case, the displayed multiplicities and eigenvalues
satisfy both identities in \eqref{eq:appendix-checks}.}{\small\par}

\subsection{Exceptional symmetric spaces of type I}

{\small\label{sec:appendix-exceptional}}{\small\par}

{\small We use compact forms and suppress finite central quotients
of the isotropy groups. All metrics have $\Ric=\frac{1}{2}g$. In
Table~\ref{tab:appendix-exceptional}, $\lambda^{[m]}$ denotes an
eigenvalue and its }{\small\emph{real}}{\small{} multiplicity. In particular,
the metric line, whose eigenvalue is $-\frac{1}{2}$, is not included.}{\small\par}

\begingroup 
\small
\setlength{\tabcolsep}{4pt}
\renewcommand{\arraystretch}{1.3}
%{\small\begingroup \setlength{\tabcolsep}{4pt} 
%\global\long\def\arraystretch{1.3}%
% }{\small{}%
\begin{longtable}[c]{|c|c|}
\caption{Complete exceptional type-I spectra on $S^{2}_{0}(T^{*}M)$.}\label{tab:appendix-exceptional}
\tabularnewline
\hline 
{\small Type, space, and dimension } & {\small Spectrum}\tabularnewline
\endfirsthead
\hline 
\multicolumn{2}{@{}l}{{\small\tablename~\thetable\ (continued)}}\tabularnewline
\hline 
{\small Type, space, and dimension } & {\small Spectrum}\tabularnewline
\endhead
\hline 
{\small EI: $E_{6}/Sp(4)$, $d=42$ } & {\small$(-\frac{1}{8})^{[308]},\ (\frac{1}{12})^{[594]}$}\tabularnewline
\hline 
\begin{cellvarwidth}[t]
{\small EII: $E_{6}/(SU(6)Sp(1))$}\\
{\small$d=40$ }
\end{cellvarwidth} & {\small$(-\frac{1}{6})^{[105]},\ (-\frac{1}{12})^{[189]},\ (\frac{1}{12})^{[525]}$}\tabularnewline
\hline 
\begin{cellvarwidth}[t]
{\small EIII: $E_{6}/(\operatorname{Spin}(10)U(1))$}\\
{\small$d=32$ }
\end{cellvarwidth} & \begin{cellvarwidth}[t]
{\small$(-\frac{1}{4})^{[20]},\ (-\frac{1}{6})^{[45]},$}\\
{\small$0^{[210]},\ (\frac{1}{12})^{[252]}$}
\end{cellvarwidth}\tabularnewline
\hline 
{\small EIV: $E_{6}/F_{4}$, $d=26$ } & {\small$(-\frac{1}{4})^{[26]},\ (\frac{1}{24})^{[324]}$}\tabularnewline
\hline 
{\small EV: $E_{7}/SU(8)$, $d=70$ } & {\small$(-\frac{1}{9})^{[720]},\ (\frac{1}{18})^{[1764]}$}\tabularnewline
\hline 
\begin{cellvarwidth}[t]
{\small EVI: $E_{7}/(\operatorname{Spin}(12)Sp(1))$}\\
{\small$d=64$ }
\end{cellvarwidth} & {\small$(-\frac{1}{6})^{[198]},\ (-\frac{1}{18})^{[495]},\ (\frac{1}{18})^{[1386]}$}\tabularnewline
\hline 
{\small EVII: $E_{7}/(E_{6}U(1))$, $d=54$ } & \begin{cellvarwidth}[t]
{\small$(-\frac{2}{9})^{[54]},\ (-\frac{1}{6})^{[78]},$}\\
{\small$0^{[650]},\ (\frac{1}{18})^{[702]}$}
\end{cellvarwidth}\tabularnewline
\hline 
\begin{cellvarwidth}[t]
{\small EVIII: $E_{8}/\operatorname{Spin}(16)$}\\
{\small$d=128$ }
\end{cellvarwidth} & {\small$(-\frac{1}{10})^{[1820]},\ (\frac{1}{30})^{[6435]}$}\tabularnewline
\hline 
{\small EIX: $E_{8}/(E_{7}Sp(1))$, $d=112$ } & {\small$(-\frac{1}{6})^{[399]},\ (-\frac{1}{30})^{[1539]},\ (\frac{1}{30})^{[4389]}$}\tabularnewline
\hline 
{\small FI: $F_{4}/(Sp(3)Sp(1))$, $d=28$ } & {\small$(-\frac{1}{6})^{[63]},\ (-\frac{1}{9})^{[90]},\ (\frac{1}{9})^{[252]}$}\tabularnewline
\hline 
{\small FII: $F_{4}/\operatorname{Spin}(9)$, $d=16$ } & {\small$(-\frac{5}{18})^{[9]},\ (\frac{1}{18})^{[126]}$}\tabularnewline
\hline 
{\small G: $G_{2}/SO(4)$, $d=8$ } & {\small$(-\frac{1}{4})^{[5]},\ (-\frac{1}{6})^{[9]},\ (\frac{1}{4})^{[21]}$}\tabularnewline
\hline 
\end{longtable}
\par
\endgroup
%\end{longtable}}{\small\endgroup}{\small\par}

{\small We give the representation data that recover every entry. Write
$V_{\lambda}$ for the complex irreducible module of highest weight
$\lambda$, using Bourbaki numbering, and set 
\[
q_{\mathfrak{k}}(\lambda)=\langle\lambda,\lambda+2\delta_{\mathfrak{k}}\rangle,
\]
where $\delta_{\mathfrak{k}}$ is half the sum of the positive roots
and long roots have squared length two. For the simple isotropy factors
appearing below, the Casimir contribution is $\frac{q_{\mathfrak{k}}(\lambda)}{2h^{\vee}_{\mathfrak{g}}}$,
with 
\[
h^{\vee}_{E_{6}}=12,\quad h^{\vee}_{E_{7}}=18,\quad h^{\vee}_{E_{8}}=30,\quad h^{\vee}_{F_{4}}=9,\quad h^{\vee}_{G_{2}}=4.
\]
The short-root factor in $G_{2}/SO(4)$ has an additional factor $\frac{1}{3}$;
the central contributions in the Hermitian cases are recorded separately.
Equation~\eqref{eq:appendix-casimir} converts a Casimir value $c$
into the curvature eigenvalue $\frac{c-1}{2}$.}{\small\par}

{\small\smallskip{}
\noindent }{\small\emph{The five spaces with two eigenvalues.}}{\small{}
For EI, EIV, EV, EVIII, and FII, the complex isotropy modules have
highest weights $\omega_{4}$ of $C_{4}$, $\omega_{4}$ of $F_{4}$,
$\omega_{4}$ of $A_{7}$, $\omega_{8}$ of $D_{8}$, and $\omega_{4}$
of $B_{4}$, respectively. Their symmetric squares have one trivial
summand, the metric, and the following two nontrivial summands: 
\[
{
\global\long\def\arraystretch{1.5}%
\begin{array}{|c|c|c|}
\hline \text{Type} & \text{highest weights in }S^{2}_{0}\mathfrak{p}_{\mathbb{C}} & (\dim V_{\lambda},q_{\mathfrak{k}}(\lambda))\\[3pt]
\hline \mathrm{EI} & 2\omega_{2},\ 2\omega_{4} & (308,18),\ (594,28)\\[3pt]
\hline \mathrm{EIV} & \omega_{4},\ 2\omega_{4} & (26,12),\ (324,26)\\[3pt]
\hline \mathrm{EV} & \omega_{2}+\omega_{6},\ 2\omega_{4} & (720,28),\ (1764,40)\\[3pt]
\hline \mathrm{EVIII} & \omega_{4},\ 2\omega_{8} & (1820,48),\ (6435,64)\\[3pt]
\hline \mathrm{FII} & \omega_{1},\ 2\omega_{4} & (9,8),\ (126,20)\\[3pt]
\hline \end{array}}
\]
The dimensions follow from \eqref{eq:appendix-highest-weights}. For
the spin representations this is particularly explicit: 
\begin{align*}
S^{2}\Delta^{+}_{16} & =\mathbb{C}\oplus\sideset{}{^{4}}\bigwedge\mathbb{C}^{16}\oplus\sideset{}{_{+}^{8}}\bigwedge\mathbb{C}^{16},\\
S^{2}\Delta_{9} & =\mathbb{C}\oplus\mathbb{C}^{9}\oplus\sideset{}{^{4}}\bigwedge\mathbb{C}^{9}.
\end{align*}
Thus $1820=\binom{16}{4}$, $6435=\frac{1}{2}\binom{16}{8}$, and
$126=\binom{9}{4}$. Dividing the displayed Casimir values by $2h^{\vee}_{\mathfrak{g}}$
and applying \eqref{eq:appendix-casimir} gives the five asserted
spectra.}{\small\par}

{\small\smallskip{}
\noindent }{\small\emph{The four quaternionic spaces.}}{\small{} For
EII, EVI, EIX, and FI, write $\mathfrak{p}_{\mathbb{C}}=U\otimes H$,
where $H=\mathbb{C}^{2}$ is the standard $Sp(1)$ module and $U$
is, respectively, $V^{A_{5}}_{\omega_{3}}$, $V^{D_{6}}_{\omega_{6}}$,
$V^{E_{7}}_{\omega_{7}}$, and $V^{C_{3}}_{\omega_{3}}$. Their dimensions
are $20,32,56,14$, and they are symplectic. In the decompositions
\[
S^{2}U=\mathcal{A}\oplus\mathcal{B},\quad\sideset{}{^{2}}\bigwedge U=\mathbb{C}\oplus\mathcal{D},
\]
$\mathcal{A}$ is the adjoint module. The required dimensions and
Casimir values are 
\[
{
\global\long\def\arraystretch{1.5}%
\begin{array}{|c|c|c|c|}
\hline \text{Type} & (\dim\mathcal{A},q(\mathcal{A})) & (\dim\mathcal{B},q(\mathcal{B})) & (\dim\mathcal{D},q(\mathcal{D}))\\[3pt]
\hline \mathrm{EII} & (35,12) & (175,24) & (189,20)\\[3pt]
\hline \mathrm{EVI} & (66,20) & (462,36) & (495,32)\\[3pt]
\hline \mathrm{EIX} & (133,36) & (1463,60) & (1539,56)\\[3pt]
\hline \mathrm{FI} & (21,8) & (84,18) & (90,14)\\[3pt]
\hline \end{array}}
\]
For reproducibility, the highest weights of $(\mathcal{A},\mathcal{B},\mathcal{D})$
in these four rows are 
\[
{
\global\long\def\arraystretch{1.5}%
\begin{array}{|c|c|c|c|}
\hline \text{Type} & \mathcal{A} & \mathcal{B} & \mathcal{D}\\[3pt]
\hline \mathrm{EII} & \omega_{1}+\omega_{5} & 2\omega_{3} & \omega_{2}+\omega_{4}\\[3pt]
\hline \mathrm{EVI} & \omega_{2} & 2\omega_{6} & \omega_{4}\\[3pt]
\hline \mathrm{EIX} & \omega_{1} & 2\omega_{7} & \omega_{6}\\[3pt]
\hline \mathrm{FI} & 2\omega_{1} & 2\omega_{3} & 2\omega_{2}.\\[3pt]
\hline \end{array}}
\]
Using $\bigwedge^{2}H=\mathbb{C}$ and deleting the metric gives 
\[
S^{2}_{0}\mathfrak{p}_{\mathbb{C}}=(\mathcal{A}\otimes S^{2}H)\oplus(\mathcal{B}\otimes S^{2}H)\oplus\mathcal{D}.
\]
Since $\dim S^{2}H=3$ and $q_{A_{1}}(2\omega_{1})=4$, the three
Casimir values are 
\[
\frac{q(\mathcal{A})+4}{2h^{\vee}_{\mathfrak{g}}},\quad\frac{q(\mathcal{B})+4}{2h^{\vee}_{\mathfrak{g}}},\quad\frac{q(\mathcal{D})}{2h^{\vee}_{\mathfrak{g}}},
\]
with multiplicities $3\dim\mathcal{A}$, $3\dim\mathcal{B}$, and
$\dim\mathcal{D}$. This gives each of the four three-eigenvalue rows.}{\small\par}

{\small\smallskip{}
\noindent }{\small\emph{The two Hermitian spaces.}}{\small{} Subscripts
now denote central charges. For EIII, $\mathfrak{p}_{\mathbb{C}}=\Delta^{+}_{1}\oplus\Delta^{-}_{-1}$,
where the half-spin modules have complex dimension sixteen. The Casimir
value on a $D_{5}$ module of charge $q$ is 
\[
c(\lambda,q)=\frac{q_{D_{5}}(\lambda)}{24}+\frac{q^{2}}{32}.
\]
The central coefficient is fixed by the tangent Casimir $\frac{1}{2}$,
since $q_{D_{5}}(\omega_{5})=\frac{45}{4}$. The decompositions 
\begin{align*}
\Delta^{+}\otimes\Delta^{-} & =\mathbb{C}\oplus V_{\omega_{2}}\oplus V_{\omega_{4}+\omega_{5}},\\
S^{2}\Delta^{+} & =V_{\omega_{1}}\oplus V_{2\omega_{5}}
\end{align*}
and their conjugates give, after removing the metric, 
\[
{
\global\long\def\arraystretch{1.5}%
\begin{array}{|c|r|c|}
\hline \text{summand(s)} & \text{total dimension} & c\\[3pt]
\hline (V_{\omega_{1}})_{2}\oplus(V_{\omega_{1}})_{-2} & 20 & \frac{1}{2}\\[3pt]
\hline (V_{\omega_{2}})_{0} & 45 & \frac{2}{3}\\[3pt]
\hline (V_{\omega_{4}+\omega_{5}})_{0} & 210 & 1\\[3pt]
\hline (V_{2\omega_{5}})_{2}\oplus(V_{2\omega_{4}})_{-2} & 252 & \frac{7}{6}\\[3pt]
\hline \end{array}}
\]
Here the unnormalized $D_{5}$ Casimir values are $9,16,24,25$, respectively.}{\small\par}

{\small For EVII, take $U=V^{E_{6}}_{\omega_{1}}$, of dimension twenty-seven.
Then $\mathfrak{p}_{\mathbb{C}}=U_{1}\oplus U^{*}_{-1}$ and 
\[
c(\lambda,q)=\frac{q_{E_{6}}(\lambda)}{36}+\frac{q^{2}}{54}.
\]
Again the tangent Casimir is $\frac{1}{2}$, because $q_{E_{6}}(\omega_{1})=\frac{52}{3}$.
Now 
\begin{align*}
U\otimes U^{*} & =\mathbb{C}\oplus V_{\omega_{2}}\oplus V_{\omega_{1}+\omega_{6}},\\
S^{2}U & =V_{\omega_{6}}\oplus V_{2\omega_{1}}.
\end{align*}
The nontrivial summands are 
\[
{
\global\long\def\arraystretch{1.5}%
\begin{array}{|c|r|c|}
\hline \text{summand(s)} & \text{total dimension} & c\\[3pt]
\hline (V_{\omega_{6}})_{2}\oplus(V_{\omega_{1}})_{-2} & 54 & \frac{5}{9}\\[3pt]
\hline (V_{\omega_{2}})_{0} & 78 & \frac{2}{3}\\[3pt]
\hline (V_{\omega_{1}+\omega_{6}})_{0} & 650 & 1\\[3pt]
\hline (V_{2\omega_{1}})_{2}\oplus(V_{2\omega_{6}})_{-2} & 702 & \frac{10}{9}.\\[3pt]
\hline \end{array}}
\]
Their unnormalized $E_{6}$ Casimir values are $\frac{52}{3},24,36,\frac{112}{3}$.
In both Hermitian cases, each conjugate pair is the complexification
of one real summand. Its total complex dimension is therefore the
required real multiplicity, so both must be counted. The zero-charge
trivial summand is the metric and is removed exactly once.}{\small\par}

{\small\smallskip{}
\noindent }{\small\emph{The space $G_{2}/SO(4)$.}}{\small{} Order
the two $A_{1}$ factors as short-root and long-root, and let $V_{a}$
have highest weight $a\omega_{1}$, hence dimension $a+1$. The isotropy
module and its trace-free symmetric square are 
\begin{align*}
\mathfrak{p}_{\mathbb{C}} & =V_{3}\otimes V_{1},\\
S^{2}_{0}\mathfrak{p}_{\mathbb{C}} & =(V_{4}\otimes V_{0})\oplus(V_{2}\otimes V_{2})\oplus(V_{6}\otimes V_{2}).
\end{align*}
These dimensions are $5,9,21$. Restriction of the ambient Killing
form gives 
\[
c(a,b)=\frac{1}{8}\left(\frac{a(a+2)}{6}+\frac{b(b+2)}{2}\right).
\]
The Casimir values are $\frac{1}{2}$, $\frac{2}{3}$, and $\frac{3}{2}$.
Applying \eqref{eq:appendix-casimir} yields the last row of the table.
The factor $\frac{1}{3}$ on the short-root factor is essential to
this normalization.}{\small\par}

{\small In all twelve rows, summing the real multiplicities gives $\frac{(d-1)(d+2)}{2}$,
and summing the eigenvalues with multiplicity gives $\frac{d+2}{4}$,
as required by \eqref{eq:appendix-checks}. These provide independent
checks of the metric-line removal, the real multiplicities, and the
curvature scale.}{\small\par}

\subsection{Einstein products of spheres and complex projective spaces}

{\small\label{subsec:appendix-products-projective}}{\small\par}

{\small We give direct calculations for the models appearing in Theorem~\ref{thm:symmetric-classification}.
All metrics in this subsection are normalized by $\Ric=\frac{1}{2}g$.
An exponent in square brackets denotes real multiplicity.}{\small\par}

{\small\smallskip{}
\noindent }{\small\emph{The product decomposition.}}{\small{} Let $M=M_{1}\times\cdots\times M_{r}$
be a product of Einstein manifolds, each with Einstein constant $\frac{1}{2}$.
At a point write $V=V_{1}\oplus\cdots\oplus V_{r}$, where $d_{\alpha}=\dim V_{\alpha}$
and $n=\sum_{\alpha}d_{\alpha}$. Extend the factor metrics $g_{\alpha}$
by zero on the other factors. There is an orthogonal invariant decomposition
\[
S^{2}_{0}(V)=\left\{ \sum_{\alpha}t_{\alpha}g_{\alpha}:\sum_{\alpha}d_{\alpha}t_{\alpha}=0\right\} \oplus\bigoplus_{\alpha}S^{2}_{0}(V_{\alpha})\oplus\bigoplus_{\alpha<\beta}(V_{\alpha}\odot V_{\beta}).
\]
The first summand consists of trace-free linear combinations of the
factor metrics and has dimension $r-1$. Since $\widehat{R}(g_{\alpha})=-\Ric_{M_{\alpha}}=-\frac{1}{2}g_{\alpha}$,
the COSK is $-\frac{1}{2}\Id$ on this summand. On $S^{2}_{0}(V_{\alpha})$
it is the COSK of the $\alpha$th factor. On each mixed summand $V_{\alpha}\odot V_{\beta}$
it vanishes: a nonzero curvature component of the product has all
four indices in one factor, whereas the two indices of a mixed tensor
lie in different factors. Consequently, 
\begin{equation}
\operatorname{Spec}(\mathring{R})=\left\{ \left(-\frac{1}{2}\right)^{[r-1]},\,0^{[\sum_{\alpha<\beta}d_{\alpha}d_{\beta}]}\right\} \cup\bigcup^{r}_{\alpha=1}\operatorname{Spec}(\mathring{R}_{M_{\alpha}}),\label{eq:appendix-general-product-spectrum}
\end{equation}
where the union retains multiplicities. This formula applies to all
Einstein products with the common normalization, not only to spheres.}{\small\par}

{\small\smallskip{}
\noindent }{\small\emph{Products of round spheres.}}{\small{} Take
\[
M=\prod^{r}_{\alpha=1}S^{d_{\alpha}}\left(\frac{1}{2(d_{\alpha}-1)}\right),\quad d_{\alpha}\geq2.
\]
Here $S^{d}(k)$ denotes the round sphere of sectional curvature $k$.
A curvature tensor of constant sectional curvature $k$ satisfies
$\widehat{R}(H)=k(H-(\tr H)\Id)$, so its restriction to $S^{2}_{0}$
is $k\Id$. Formula \eqref{eq:appendix-general-product-spectrum}
therefore gives the complete spectrum 
\begin{align*}
 & \left(-\frac{1}{2}\right)^{[r-1]},\quad0^{[\sum_{\alpha<\beta}d_{\alpha}d_{\beta}]},\\
 & \left(\frac{1}{2(d_{\alpha}-1)}\right)^{[\frac{(d_{\alpha}-1)(d_{\alpha}+2)}{2}]},\quad1\leq\alpha\leq r.
\end{align*}
The multiplicities sum to $\frac{(n-1)(n+2)}{2}$. The weighted sum
of the eigenvalues is 
\[
-\frac{r-1}{2}+\sum^{r}_{\alpha=1}\frac{d_{\alpha}+2}{4}=\frac{n+2}{4},
\]
as required by $s=\frac{n}{2}$ and \eqref{eq:second-kind-trace}.}{\small\par}

{\small The shifted endpoint in Section~\ref{sec:symmetric-classification}
is particularly transparent in this spectrum. Since $\rho=\frac{1}{2}$,
\[
\mathring{W}+\beta_{n}\Id=\mathring{R}+\frac{1}{n+2}\Id.
\]
Suppose $r\geq2$ and put $k=\frac{n+2}{2}$. The first $r-1$ eigenvalues
of the shifted operator are $-\frac{1}{2}+\frac{1}{n+2}$; the next
$\sum_{\alpha<\beta}d_{\alpha}d_{\beta}$ are $\frac{1}{n+2}$. There
are enough of these to fill the first $k$ eigenvalues: $n\geq2r$
gives $k>r-1$, and 
\[
\sum_{\alpha<\beta}d_{\alpha}d_{\beta}=\frac{1}{2}\sum_{\alpha}d_{\alpha}(n-d_{\alpha})\geq n\geq k-(r-1).
\]
Thus, with the same interpretation for fractional $k$, 
\[
\Sigma_{\frac{n+2}{2}}(\mathring{W}+\beta_{n}\Id)=-\frac{r-1}{2}+\frac{k}{n+2}=\frac{2-r}{2}.
\]
In particular, every Einstein product of two round spheres lies exactly
on the boundary, while products of three or more factors lie outside
the cone. For one round sphere the shifted sum is $\frac{3n}{4(n-1)}>0$,
in agreement with Theorem~\ref{thm:symmetric-classification}.}{\small\par}

{\small\smallskip{}
\noindent }{\small\emph{Complex projective space.}}{\small{} Let $\mathbb{CP}^{m}$
carry its Fubini-Study metric, and let $J$ be the orthogonal complex
structure on a tangent space $V$ of real dimension $2m$. With holomorphic
sectional curvature $4c$, its curvature tensor in our convention
is 
\begin{align}
R(X,Y,Z,T)=c\bigl( & \langle X,Z\rangle\langle Y,T\rangle-\langle X,T\rangle\langle Y,Z\rangle\nonumber \\
 & +\langle JX,Z\rangle\langle JY,T\rangle-\langle JX,T\rangle\langle JY,Z\rangle\nonumber \\
 & +2\langle JX,Y\rangle\langle JZ,T\rangle\bigr).\label{eq:appendix-CP-curvature}
\end{align}
Its Ricci tensor is $2(m+1)c\,g$, so the normalization $\Ric=\frac{1}{2}g$
means 
\begin{equation}
c=\frac{1}{4(m+1)}.\label{eq:appendix-CP-scale}
\end{equation}
For a symmetric endomorphism $H$, contraction of \eqref{eq:appendix-CP-curvature}
gives 
\begin{equation}
\widehat{R}(H)=c\bigl(H-(\tr H)\Id+3JHJ\bigr).\label{eq:appendix-CP-contraction}
\end{equation}
To see the last term, the contraction of $\langle JX,Z\rangle\langle JY,T\rangle$
contributes $JHJ$; the contraction involving the trace of $JH$ vanishes
because $J$ is skew-adjoint and $H$ is symmetric; and the last curvature
term contributes $2JHJ$. In particular, $\widehat{R}(\Id)=-2(m+1)c\Id=-\frac{1}{2}\Id$,
in agreement with $\widehat{R}(\Id)=-\Ric$.}{\small\par}

{\small The trace-free symmetric endomorphisms split orthogonally as
\[
S^{2}_{0}(V)=\{H:HJ=JH,\ \tr H=0\}\oplus\{H:HJ=-JH\}.
\]
Orthogonality follows from the orthogonal involution $H\mapsto-JHJ$.
The first summand is the space of complex Hermitian trace-free matrices
and has real dimension $m^{2}-1$. In the second summand $H$ is represented
by a complex symmetric matrix acting antilinearly; its real dimension
is $2\dim_{\mathbb{C}}S^{2}\mathbb{C}^{m}=m(m+1)$. If $HJ=JH$, then
$JHJ=-H$, whereas if $HJ=-JH$, then $JHJ=H$. Equations~\eqref{eq:appendix-CP-scale}
and \eqref{eq:appendix-CP-contraction} now give 
\begin{equation}
\operatorname{Spec}(\mathring{R}_{\mathbb{CP}^{m}})=\left\{ \left(-\frac{1}{2(m+1)}\right)^{[m^{2}-1]},\quad\left(\frac{1}{m+1}\right)^{[m(m+1)]}\right\} .\label{eq:appendix-CP-spectrum}
\end{equation}
The multiplicities sum to $\frac{(2m-1)(2m+2)}{2}$, and the trace
is 
\[
-\frac{m^{2}-1}{2(m+1)}+\frac{m(m+1)}{m+1}=\frac{m+1}{2}=\frac{n+2}{4},\quad n=2m.
\]
}{\small\par}

{\small For the shifted operator, adding $\frac{1}{n+2}=\frac{1}{2(m+1)}$
gives 
\[
\operatorname{Spec}(\mathring{W}+\beta_{n}\Id)=\left\{ 0^{[m^{2}-1]},\quad\left(\frac{3}{2(m+1)}\right)^{[m(m+1)]}\right\} .
\]
For $m\geq2$ there are at least $m+1=\frac{n+2}{2}$ zero eigenvalues,
so 
\[
\Sigma_{\frac{n+2}{2}}(\mathring{W}+\beta_{n}\Id)=0.
\]
When $m=1$, the first multiplicity is zero and \eqref{eq:appendix-CP-spectrum}
reduces to the eigenvalue $\frac{1}{2}$ with multiplicity two. This
is precisely the round two-sphere with Einstein constant $\frac{1}{2}$.}{\small\par}

{\small Finally, \eqref{eq:appendix-general-product-spectrum} also
gives all products of complex projective spaces, or mixed products
of spheres and complex projective spaces. For example, for $\prod^{r}_{\alpha=1}\mathbb{CP}^{m_{\alpha}}$
it gives 
\begin{align*}
 & \left(-\frac{1}{2}\right)^{[r-1]},\quad0^{[4\sum_{\alpha<\beta}m_{\alpha}m_{\beta}]},\\
 & \left(-\frac{1}{2(m_{\alpha}+1)}\right)^{[m^{2}_{\alpha}-1]},\quad\left(\frac{1}{m_{\alpha}+1}\right)^{[m_{\alpha}(m_{\alpha}+1)]},\quad1\leq\alpha\leq r,
\end{align*}
where the last two entries are included for each $\alpha$ and coincident
eigenvalues are combined.}{\small\par}

\bibliographystyle{alpha}
\bibliography{references}

\end{document}